\documentclass[10pt,a4paper,twoside]{article}

\usepackage[a4paper, left=3cm, right=3cm, top = 3cm, bottom = 3cm ]{geometry}
\usepackage{latexsym,array,delarray,amsthm,amssymb,psfrag, dsfont}
\usepackage{amsmath,amscd}
\usepackage{amsthm}
\usepackage[all]{xy}
\usepackage{seqsplit}
\usepackage{enumerate}

\usepackage{stmaryrd}
\usepackage{amsmath}
\usepackage{amsfonts}
\usepackage{mathtools}
\usepackage{txfonts}
\usepackage[usenames,dvipsnames]{xcolor}
\usepackage{hyperref}
\hypersetup{colorlinks=true,citecolor=NavyBlue,linkcolor=Brown,urlcolor=Orange}

\usepackage{tikz-cd}

\usepackage[utf8]{inputenc}
\usepackage{graphicx}
\usetikzlibrary{matrix, calc, arrows}
\usepackage{scalerel}

\newcommand\reallywidehat[1]{\arraycolsep=0pt\relax%
\begin{array}{c}
\stretchto{
  \scaleto{
    \scalerel*[\widthof{\ensuremath{#1}}]{\kern-.5pt\bigwedge\kern-.5pt}
    {\rule[-\textheight/2]{1ex}{\textheight}} 
  }{\textheight} %
}{0.5ex}\\           
#1\\                 
\rule{-1ex}{0ex}
\end{array}
}
\DeclareMathOperator{\age}{age}
\DeclareMathOperator{\CH}{CH}
\DeclareMathOperator{\CHM}{CHM}

\DeclareMathOperator{\DM}{DM}
\newcommand{\h}{\ensuremath\mathfrak{h}}
\DeclareMathOperator{\Hom}{Hom}
\DeclareMathOperator{\orb}{orb}
\DeclareMathOperator{\rk}{rk}

\newtheorem{theorem}{Theorem}[section]
\newtheorem{conj}[theorem]{Conjecture}
\newtheorem{lemma}[theorem]{Lemma}
\newtheorem{proposition}[theorem]{Proposition}

\theoremstyle{definition}
\newtheorem{definition}[theorem]{Definition}

\theoremstyle{remark}
\newtheorem{remark}[theorem]{Remark}
\newtheorem*{theorem*}{Theorem}

\numberwithin{equation}{section}

\def\YEAR{\year}\newcount\VOL\VOL=\YEAR\advance\VOL by-1995
\def\firstpage{1}\def\lastpage{1000}
\def\received{}\def\revised{}
\def\communicated{}

\makeatletter
\def\magnification{\afterassignment\m@g\count@}
\def\m@g{\mag=\count@\hsize6.5truein\vsize8.9truein\dimen\footins8truein}
\makeatother

\font\eightrm=cmr8
\font\caps=cmcsc10                    
\font\Caps=cmcsc10 scaled \magstep1   
\font\scaps=cmcsc8

\makeatletter
\@specialpagefalse\headheight=8.5pt
\def\DocMath{{\def\th{\thinspace}\scaps Documenta Math.}}
\renewcommand{\@oddfoot}{\hfill\scaps Documenta Mathematica 
    \number\VOL\  (\number\YEAR) \number\firstpage--\lastpage\hfill}
\renewcommand{\@evenfoot}{\ifnum\thepage>\lastpage\hfill\scaps
    Documenta Mathematica \number\VOL\  (\number\YEAR)\hfill\else\@oddfoot\fi}%

\renewcommand{\@evenhead}{%
    \ifnum\thepage>\lastpage\rlap{\thepage}\hfill%
    \else\rlap{\thepage}\slshape\leftmark\hfill{\caps\SAuthor}\hfill\fi}%
\renewcommand{\@oddhead}{%
    \ifnum\thepage=\firstpage{\DocMath\hfill\llap{\thepage}}%
    \else{\slshape\rightmark}\hfill{\caps\STitle}\hfill\llap{\thepage}\fi}%
\makeatother

\def\TSkip{\bigskip}
\newbox\TheTitle{\obeylines\gdef\GetTitle #1
\ShortTitle  #2
\SubTitle    #3
\Author      #4
\ShortAuthor #5
\EndTitle
{\setbox\TheTitle=\vbox{\baselineskip=20pt\let\par=\cr\obeylines%
\halign{\centerline{\Caps##}\cr\noalign{\medskip}\cr#1\cr}}%
	\copy\TheTitle\TSkip\TSkip%
\def\next{#2}\ifx\next\empty\gdef\STitle{#1}\else\gdef\STitle{#2}\fi%
\def\next{#3}\ifx\next\empty%
    \else\setbox\TheTitle=\vbox{\baselineskip=20pt\let\par=\cr\obeylines%
    \halign{\centerline{\caps##} #3\cr}}\copy\TheTitle\TSkip\TSkip\fi%
\centerline{\caps #4}\TSkip\TSkip%
\def\next{#5}\ifx\next\empty\gdef\SAuthor{#4}\else\gdef\SAuthor{#5}\fi%
\ifx\received\empty\relax
    \else\centerline{\eightrm Received: \received}\fi%
\ifx\revised\empty\TSkip%
    \else\centerline{\eightrm Revised: \revised}\TSkip\fi%
\ifx\communicated\empty\relax
    \else\centerline{\eightrm Communicated by \communicated}\fi\TSkip\TSkip%
\catcode'015=5}}\def\Title{\obeylines\GetTitle}
\def\Abstract{\begingroup\narrower
    \parskip=\medskipamount\parindent=0pt{\caps Abstract. }}
\def\EndAbstract{\par\endgroup\TSkip}

\long\def\MSC#1\EndMSC{\def\arg{#1}\ifx\arg\empty\relax\else
     {\par\narrower\noindent%
     2010 Mathematics Subject Classification: #1\par}\fi}

\long\def\KEY#1\EndKEY{\def\arg{#1}\ifx\arg\empty\relax\else
	{\par\narrower\noindent Keywords and Phrases: #1\par}\fi\TSkip}

\newbox\TheAdd\def\Addresses{\vfill\copy\TheAdd\vfill
    \ifodd\number\lastpage\vfill\eject\phantom{.}\vfill\eject\fi}
{\obeylines\gdef\GetAddress #1
\Address #2 
\Address #3
\Address #4
\EndAddress
{\def\xs{4.3truecm}\parindent=0pt
\setbox0=\vtop{{\obeylines\hsize=\xs#1\par}}\def\next{#2}
\ifx\next\empty 
     \setbox\TheAdd=\hbox to\hsize{\hfill\copy0\hfill}
\else\setbox1=\vtop{{\obeylines\hsize=\xs#2\par}}\def\next{#3}
\ifx\next\empty 
     \setbox\TheAdd=\hbox to\hsize{\hfill\copy0\hfill\copy1\hfill}
\else\setbox2=\vtop{{\obeylines\hsize=\xs#3\par}}\def\next{#4}
\ifx\next\empty\ 
     \setbox\TheAdd=\vtop{\hbox to\hsize{\hfill\copy0\hfill\copy1\hfill}
                \vskip20pt\hbox to\hsize{\hfill\copy2\hfill}}
\else\setbox3=\vtop{{\obeylines\hsize=\xs#4\par}}
     \setbox\TheAdd=\vtop{\hbox to\hsize{\hfill\copy0\hfill\copy1\hfill}
	        \vskip20pt\hbox to\hsize{\hfill\copy2\hfill\copy3\hfill}}
\fi\fi\fi\catcode'015=5}}\gdef\Address{\obeylines\GetAddress}

\def\LOCAL{\jobname.files}

\begin{document}
\Title
Orbifold products for higher K-theory and motivic cohomology
\ShortTitle 
Orbifold products for K-theory and motivic cohomology
\SubTitle   
\Author 
Lie Fu and Manh Toan Nguyen
\ShortAuthor 
\EndTitle
\Abstract 
Due to the work of many authors in the last decades, given an algebraic orbifold (smooth proper Deligne--Mumford stack with trivial generic stabilizer), one can construct its orbifold Chow ring and orbifold Grothendieck ring, and relate them by the orbifold Chern character map, generalizing the fundamental work of Chen--Ruan on orbifold cohomology. In this paper, we extend this theory naturally to higher Chow groups and higher algebraic $K$-theory, mainly following the work of Jarvis--Kaufmann--Kimura and Edidin--Jarvis--Kimura.
\EndAbstract
\MSC 
19E08, 19E15, 14C15, 55N32.
\EndMSC
\KEY 
Orbifold cohomology, K-theory, motivic cohomology, Chow rings, hyper-K\"ahler resolution.
\EndKEY
\Address 
Institut Camille Jordan
Universit\'e Lyon 1
France\\
\&\\
IMAPP
Radboud University
Netherlands
\Address
Institut für Mathematik
Universität Osnabrück
Germany
\Address
\Address
\EndAddress
\setcounter{tocdepth}{1}
\tableofcontents
\section{Introduction}

In their seminal papers \cite{MR1950941} and \cite{MR2104605}, Chen and Ruan constructed the \emph{orbifold cohomology} theory for complex orbifolds. More precisely, given an orbifold $\mathcal{X}$, there is a rationally graded, associative and super-commutative algebra whose underlying vector space is the cohomology of the \emph{inertia orbifold} of $\mathcal{X}$. The highly non-trivial multiplicative structure of the orbifold cohomology ring is  defined using orbifold Gromov--Witten theory, in particular, the construction of the virtual fundamental class of some moduli stack of the so-called \emph{ghost} stable maps, which are stable maps (from orbifoldal curves) of degree 0, thus invisible if the orbifold $\mathcal{X}$ is a manifold.

This striking new theory attracted a lot of interests and was  revisited repeatedly by various mathematicians. In this paper, we restrict our attention to \emph{algebraic orbifolds}, namely, smooth Deligne--Mumford stacks with trivial generic stabilizer and projective coarse moduli space.  On one hand, Fantechi--G\"ottsche \cite{FG1} and Lehn--Sorger \cite{LS} gave a simplification of the construction of the orbifold cohomology in the case of global quotients of projective complex manifolds by finite groups\,; on the other hand, Abramovich--Graber--Vistoli \cite{AGV1} \cite{MR2450211}, based on \cite{MR1862797}, provided a general algebro-geometric construction (\emph{i.e.} in the language of stacks) of the orbifold cohomology ring and actually the orbifold Chow ring \footnote{The work \cite{AGV1} takes care of the general Gromov--Witten theory of smooth Deligne--Mumford stacks, in particular, their quantum cohomology / Chow ring. The case of orbifold cohomology / Chow ring is obtained by simply taking the degree-zero part.}. 

A common feature of the aforementioned works is to construct some obstruction vector bundles by using some moduli space of curves or that of stable maps from curves. In contrast, for global quotients of projective complex manifolds by finite groups, Jarvis--Kaufmann--Kimura \cite{JKK1} furnished a purely combinatorial definition of the \emph{class} of the obstruction vector bundle in the Grothendieck group without appealing to moduli spaces of curves, and hence gave a much more elementary construction of the orbifold theories (cohomology ring, Chow ring, Grothendieck ring, topological K-theory \emph{etc.}). Their construction involves only the fixed loci of the group elements and various normal bundles between them (together with the naturally endowed actions). In a subsequent work, Edidin--Jarvis--Kimura \cite{EJK1} extended the construction in \cite{JKK1} to all smooth Deligne--Mumford stacks which are quotients of projective complex manifolds by linear algebraic groups, using the so-called logarithmic trace and twisted pull-backs (see \S \ref{sect:General}). Let us briefly summarize their results. In the following, we use exclusively rational coefficients\,; $\CH^{*}_{G}(-)$ is the equivariant Chow group of Totaro \cite{Totaro} and Edidin--Graham \cite{EG1}, and $K^{G}_{0}(-):=K_{0}\left([-/G]\right)$ is the equivariant Grothendieck group, namely, the Grothendieck group of the category of $G$-equivariant vector bundles. Here is the main result of \cite{JKK1} and \cite{EJK1}.

\begin{theorem}[Jarvis--Kaufmann--Kimura \cite{JKK1} and  Edidin--Jarvis--Kimura \cite{EJK1}]\label{thm:EJKK}
Let $X$ be a smooth projective complex variety endowed with an action of a linear algebraic group $G$. Denote by $I_{G}(X):=\left\{(g,x)~\vert~gx=x\right\}$ the inertia variety, endowed with a natural $G$-action given by $h.(g,x)=(hgh^{-1}, hx)$ for all $h\in G$ and $(g,x)\in I_{G}(X)$. Assume that the action has finite stabilizer, \emph{i.e.} $I_{G}(X)\to X$ is a finite morphism. Let $\mathcal{X}:=[X/G]$ denote the quotient Deligne--Mumford stack and $I \mathcal{X}: = [I_G(X)/G]$ its inertia stack.
Then 
\begin{enumerate}[$(i)$]
\item On the equivariant Chow group $\CH^{*}_{G}\left(I_{G}(X)\right)$, there is an orbifold product $\star_{c_{\mathbb T}}$, which makes it into a commutative and associative graded ring.
\item The graded ring $\CH^{*}_{G}\left(I_{G}(X)\right)$, endowed with the orbifold product, is independent of the choice of the presentation of the stack $\mathcal{X}$ and coincides with the product defined in Abramovich--Graber--Vistoli \cite{AGV1}. Hence it is called the \emph{orbifold Chow ring} of $\mathcal{X}$ and denoted by $\CH^{*}_{\orb}(\mathcal{X})$. 
\item On the equivariant Grothendieck group $K_{0}^{G}\left(I_{G}(X)\right)=K_{0}(I\mathcal{X})$, there is an orbifold product $\star_{\mathcal{E}_{\mathbb T}}$, which makes it into a commutative and associative ring.
\item The ring $K_{0}^{G}\left(I_{G}(X)\right)$ endowed with the orbifold product is independent of the choice of the presentation of the stack $\mathcal{X}$ and is called the (full) \emph{orbifold Grothendieck ring} of $\mathcal{X}$, denoted by $K^{\orb}_{0}(\mathcal{X})$.
\item There is a natural ring homomorphism with respect to the above orbifold products, called the \emph{orbifold Chern character map}, 
$$\mathfrak{ch}: K_{0}^{G}\left(I_{G}(X)\right)\longrightarrow \CH^{*}_{G}\left(I_{G}(X)\right).$$
It induces an isomorphism 
$$
\mathfrak{ch}: K_0^{G}(I_G(X))^{\wedge} \xrightarrow{\simeq}\CH^*_{G}(I_G(X)),
$$
where the left-hand side is the completion with respect to the augmentation ideal of the representation ring of $G$.
\end{enumerate}
\end{theorem}

The work \cite{JKK1} treated the situation where $G$ is a finite group. In that case, $\CH^{*}_{G}\left(I_{G}(X)\right)$ and $K_{0}^{G}\left(I_{G}(X)\right)^{\wedge}$ are simply the $G$-invariant parts of the larger spaces $\CH^{*}\left(I_{G}(X)\right)$ and $K_{0}\left(I_{G}(X)\right)$ respectively, where the orbifold products are  already defined, giving rise to the so-called \emph{stringy Chow / Grothendieck rings}.

\subsection{Orbifold higher Chow ring and higher K-theory}

The first main purpose of this article is to extend the work of Jarvis--Kaufmann--Kimura \cite{JKK1} and  Edidin--Jarvis--Kimura \cite{EJK1} for Bloch's \emph{higher} Chow groups \cite{Bloch1} (or equivalently, the motivic cohomology \cite{Voevodsky1}) and for Quillen's \emph{higher} algebraic K-theory \cite{DQ1}, by proving the following analogue of Theorem \ref{thm:EJKK}. Similarly as before, $\CH^{*}_{G}(-, \bullet)$ is the rational equivariant higher Chow group of Edidin--Graham \cite{EG1} (see \S\ref{subsect:EquiChow}), and $K^{G}_{\bullet}(-):=K_{\bullet}\left([-/G]\right)$ is the rational equivariant algebraic K-theory (see \S\ref{subsect:EquiK}), namely, Quillen's K-theory \cite{DQ1} of the exact category of $G$-equivariant vector bundles.

\begin{theorem}\label{thm:main1}
Assumptions and notations are as in Theorem \ref{thm:EJKK}.
\begin{enumerate}[$(i)$]
\item On the equivariant higher Chow group $\CH^{*}_{G}\left(I_{G}(X), \bullet\right)$, there is an orbifold product $\star_{c_{\mathbb T}}$, which makes it into a (graded) commutative and associative bigraded ring.
\item The bigraded ring $\CH^{*}_{G}\left(I_{G}(X), \bullet\right)$, endowed with the orbifold product, is independent of the choice of the presentation of the stack $\mathcal{X}$. 
\item On the equivariant algebraic K-theory $K_{\bullet}^{G}\left(I_{G}(X)\right)$, there is an orbifold product $\star_{\mathcal{E}_{\mathbb T}}$, which makes it into a (graded) commutative and associative graded ring.
\item The graded ring $K_{\bullet}^{G}\left(I_{G}(X)\right)$ endowed with the orbifold product is independent of the choice of the presentation of the stack $\mathcal{X}$.
\item There is a natural graded ring homomorphism with respect to the orbifold products, called the \emph{orbifold (higher) Chern character map}, 
$$\mathfrak{ch}: K_{\bullet}^{G}\left(I_{G}(X)\right)\longrightarrow \CH^{*}_{G}\left(I_{G}(X), \bullet\right).$$
It induces an isomorphism 
$$
\mathfrak{ch}: K_\bullet^{G}(I_G(X))^{\wedge} \xrightarrow{\simeq}\CH^*_{G}(I_G(X), \bullet),
$$
where the left-hand side is the completion with respect to the augmentation ideal of the representation ring of $G$.
\end{enumerate}
\end{theorem}
Here the ring $\CH^{*}_{G}\left(I_{G}(X), \bullet\right)$ (\emph{resp.} $K_{\bullet}^{G}\left(I_{G}(X)\right)$) is called the \emph{orbifold higher Chow ring} (\emph{resp.} \emph{orbifold K-theory}) of the stack $\mathcal{X}$ and denoted by $\CH^{*}_{\orb}(\mathcal{X}, \bullet)$ (\emph{resp.} $K^{\orb}_{\bullet}(\mathcal{X})$).
As in \cite{JKK1}, when $G$ is finite, the orbifold products on $\CH^{*}_{G}\left(I_{G}(X), \bullet\right)$ and $K_{\bullet}^{G}\left(I_{G}(X)\right)^{\wedge}$ extend to the larger spaces $\CH^{*}\left(I_{G}(X), \bullet\right)$ and $K_{\bullet}\left(I_{G}(X)\right)$ respectively, giving rise to \emph{stringy motivic cohomology} and \emph{stringy K-theory}, see \S \ref{sect:FiniteGroupQuotient} for the details.

\subsection{Orbifold motives}
In the first author's joint work with Zhiyu Tian and Charles Vial \cite{MHRCKummer}, for a global quotient of a smooth projective variety by a finite group, its \emph{orbifold Chow motive} is constructed, following the strategy of \cite{JKK1}, as a commutative algebra object in the category of rational Chow motives $\CHM_{\mathbb Q}$.

We construct here for any Deligne--Mumford stack $\mathcal{X}$ which is the global quotient of a smooth projective variety by a linear algebraic group, an algebra object $M_{\orb}(\mathcal{X})$ in the category of mixed motives with rational coefficients $\DM_{\mathbb Q}$. Precisely, it is the motive of its inertia stack $M(\mathcal{I}_\mathcal{X})$ together with an algebra structure given by the orbifold product. See \S \ref{subsect:OrbMot} for some details.

\subsection{Hyper-K\"ahler resolution conjectures}
Inspired by the topological string theory, one of the most important motivations (proposed by Ruan \cite{MR2234886}) to introduce these orbifold theories of a Deligne--Mumford stack is to relate it to the corresponding ``ordinary'' theories of the (crepant) resolutions of the singular coarse moduli space. More precisely, we have the following series of Hyper-K\"ahler Resolution Conjectures (HRC). Recall that a projective manifold is called \emph{hyper-K\"ahler} if it admits a holomorphic symplectic 2-form\footnote{The manifolds thus defined should rather be called \emph{holomorphic symplectic}, and hyper-K\"ahler varieties (also known as \emph{irreducible holomorphic symplectic varieties}) in the literature are the simply-connected holomorphic symplectic varieties such that symplectic form is unique up to scalar, see \cite{MR730926}, \cite{MR1664696}, \cite{MR1963559} \emph{etc.} However, in this paper we will abuse slightly the language and call holomorphic symplectic varieties hyper-K\"ahler.}.
\begin{conj}[Hyper-K\"ahler Resolution Conjectures]\label{conj:HRC}
Let $\mathcal{X}$ be a smooth algebraic orbifold, with coarse moduli space $|\mathcal{X}|$. Suppose there is a crepant resolution $Y\to |\mathcal{X}|$ with $Y$ being hyper-K\"ahler, then we have 
\begin{enumerate}[$(i)$]
\item \textit{(Cohomological HRC \cite{MR2234886})} an isomorphism of graded commutative $\mathbb{C}$-algebras\,: $$H^*(Y, \mathbb{C})\simeq H^*_{\orb}(\mathcal{X},\mathbb{C}).$$
\item \textit{(K-theoretic HRC \cite{JKK1})} an isomorphism of commutative $\mathbb{C}$-algebras\,: $$K_{0}(Y)_{\mathbb{C}}\simeq {K^{\orb}_{0}}(\mathcal{X})_{\mathbb{C}}^{\wedge}.$$
\item \textit{(Chow-theoretic HRC \cite{MHRCKummer})} an isomorphism of commutative graded $\mathbb{C}$-algebras\,: $$\CH^{*}(Y)_{\mathbb{C}}\simeq \CH_{\orb}^{*}(\mathcal{X})_{\mathbb{C}}.$$
\item \textit{(Motivic HRC \cite{MHRCKummer})} an isomorphism of commutative algebra objects in the category of complex mixed motives $\DM_{\mathbb{C}}$\,: $$M(Y)\simeq M_{\orb}(\mathcal{X}).$$
\end{enumerate}
\end{conj}
Note that in Conjecture \ref{conj:HRC}, $(ii)$ and $(iii)$ are equivalent by the Chern character map defined in \cite{JKK1} (see $(v)$ of Theorem \ref{thm:EJKK}), and the motivic version $(iv)$ implies the others \cite{MHRCKummer}.

With orbifold higher Chow rings and orbifold higher algebraic K-theory being defined in this paper (Theorem \ref{thm:main1}), we now propose to complete Conjecture \ref{conj:HRC} by including the ``higher'' invariants in $(ii)$ and $(iii)$\,:

\begin{conj}[Hyper-K\"ahler Resolution Conjectures: strengthened] \label{conj:HRC+}
Hypotheses and conclusions are as in Conjecture \ref{conj:HRC}, except that $(ii)$ and $(iii)$ are respectively replaced by\\
$(ii)^{+}$ (K-theoretic HRC) an isomorphism of commutative graded $\mathbb{C}$-algebras\,: $$K_{\bullet}(Y)_{\mathbb{C}}\simeq K^{\orb}_{\bullet}(\mathcal{X})^{\wedge}_{\mathbb{C}}.$$
$(iii)^{+}$ \textit{(Chow-theoretic HRC)} an isomorphism of commutative bigraded $\mathbb{C}$-algebras\,: $$\CH^{*}(Y, \bullet)_{\mathbb{C}}\simeq \CH_{\orb}^{*}(\mathcal{X}, \bullet)_{\mathbb{C}}.$$
\end{conj}

On one hand, by Theorem \ref{thm:main1} on the orbifold (higher) Chern character map, $(iii)^{+}$ is equivalent to $(ii)^{+}$ (Lemma \ref{lemma:KChowEquiv})\,; on the other hand, we will show in Proposition \ref{prop:MotiveImpliesChow} that the implication $(iv)\Longrightarrow (iii)^{+}$ holds. As a consequence, we can improve the first author's previous joint works with Tian and Vial \cite{MHRCKummer}, \cite{MHRCK3} and \cite{McKaySurface} by including the higher K-theory and higher Chow groups, thus confirming the (strengthened) hyper-K\"ahler resolution conjecture in several interesting cases\,:

\begin{theorem}\label{thm:main2}
Conjecture \ref{conj:HRC+} holds in the following cases, where $n\in \mathbb{N}$, $A$ is an abelian surface and $S$ is a projective K3 surface.
\begin{enumerate}[$(i)$]
\item $\mathcal{X}=[A^{n}/\mathfrak{S}_{n}]$, $Y=A^{[n]}$ the $n$-th Hilbert scheme of points of $A$ and the resolution is the Hilbert--Chow morphism.
\item $\mathcal{X}=[A^{n+1}_{0}/\mathfrak{S}_{n+1}]$, $Y=K_{n}(A)$ the $n$-th generalized Kummer variety and the resolution is the restriction of the Hilbert--Chow morphism, where $A^{n+1}_0$ denotes the kernel of the summation map $A^{n+1}\to A$, endowed with the natural action of $\mathfrak{S}_{n+1}$.
\item $\mathcal{X}=[S^{n}/\mathfrak{S}_{n}]$, $Y=S^{[n]}$ the $n$-th Hilbert scheme of points of $S$  and the resolution is the Hilbert--Chow morphism.
\item $\mathcal{X}$ is a  2-dimensional algebraic orbifold with isolated stacky points and $Y$ is the minimal resolution of $|\mathcal{X}|$.
\end{enumerate}
\end{theorem}

The cohomological hyper-K\"ahler resolution conjecture was proved in the cases $(i)$ and $(iii)$ by Fantechi--G\"ottsche \cite{FG1} and Lehn--Sorger \cite{LS}. Conjecture \ref{conj:HRC} was proved in the cases $(i)$ and $(ii)$ in \cite{MHRCKummer},  in the case $(iii)$ in \cite{MHRCK3} and in the case $(iv)$ in \cite{McKaySurface}.

\subsection{Notation and Convention}
We denote $\mathbf{Sch}_k$ for the category of separated noetherian schemes which are quasi-projective over a field $k$. 
The full subcategory of $\mathbf{Sch}_k$ consisting of smooth varieties will be denoted by $\mathbf{Sm}_k$. 
Chow groups and K-theory are with rational coefficients. 

If $X \to Y$ is a smooth morphism in $\mathbf{Sch}_k$, the relative tangent bundle will be denoted by $T_{X/Y}$. When $Y = \operatorname{Spec} (k)$, we write simply $TX$ or $T_X$. 

\bigskip
\noindent \textbf{Acknowledgement\,:} This work was started when both authors were participating the 2017 Trimester \emph{K-theory and related fields} at the Hausdorff Research Institute for Mathematics in Bonn. We thank the institute for the hospitality and the exceptional working condition.

 Lie Fu is supported by ECOVA (ANR-15-CE40-0002), HodgeFun (ANR-16-CE40-0011), LABEX MILYON (ANR-10-LABEX-0070) of Universit\'e de Lyon and \emph{Projet Inter-Laboratoire} 2017, 2018, 2019 by F\'ed\'eration de Recherche en Math\'ematiques Rh\^one-Alpes/Auvergne CNRS 3490. Manh Toan Nguyen is supported by Universität Osnabrück, the DFG - GK 1916 \textit{Kombinatorische Strukturen in der Geometrie}, and the DFG - SPP 1786 \textit{Homotopy Theory and Algebraic Geometry}.

\section{Preliminaries on K-theory and motivic cohomology}
In this section, some fundamental results in algebraic K-theory and motivic cohomology are collected for later use.

\subsection{Algebraic K-theory}\label{subsect:Ktheory}

For any $X \in \mathbf{Sch}_k$, let $K(X)$ be the connected $K$-theory spectrum of the exact category of vector bundles on $X$ in the sense of Quillen \cite{DQ1}. 
This is homotopy equivalent to the Thomason--Trobaugh's connected $K$-theory spectrum of the complicial bi-Waldhausen category of perfect complexes on $X$ \cite[Proposition 3.10]{TT1}.
We will allow ourselves to identity these two constructions. 
The $i$-th $K$-group of $X$, denoted by $K_i(X)$, is by definition the $i$-th homotopy group of $K(X)$. In particular, $K_0(X)$ is the Grothendieck group of the category of vector bundles on $X$.
We set
$$
K_{\bullet}(X):= \bigoplus_{i} K_i(X).
$$

The assignment $X \mapsto K(X)$ (and hence $X \mapsto K_{\bullet}(X)$) is a contravariant functor on $\mathbf{Sch}_k$ and 
a covariant functor on the category of quasi-projective schemes of finite type over $k$ with proper maps of finite Tor-dimension \cite[3.14 and 3.16.2 - 3.16.6]{TT1}. 
The tensor product of vector bundles over $\mathcal{O}_X$ induces a pairing
$$
\otimes: K(X) \wedge K(X) \to K(X)
$$
which is commutative and associative up to ``coherent homotopy''. This makes $K_{\bullet}(X)$ a graded commutative ring with unit $[\mathcal{O}_X] \in K_0(X)$. We will use the notation '$\cup$' for this product.

\begin{proposition}[Projection formula {\cite[Proposition 2.10]{DQ1}, \cite[Proposition 3.17]{TT1}}] \label{Projection formula}
Suppose that $X, Y \in \mathbf{Sch}_k$ and $f: X \to Y$ is proper morphism of finite Tor-dimension. The following diagram 
\[
\begin{tikzpicture}[commutative diagrams/every diagram]
\node (P0) at (90:2cm) {$K(Y) \wedge K(X)$};
\node (P1) at (90+72:1.5cm) {$K(X) \wedge K(X)$} ;
\node (P2) at (90+2*72:1.5cm) {\makebox[5ex][r]{$K(X)$}};
\node (P3) at (90+3*72:1.5cm) {\makebox[5ex][l]{$K(Y)$}};
\node (P4) at (90+4*72:1.5cm) {$K(Y) \wedge K(Y)$};
\path[commutative diagrams/.cd, every arrow, every label]
(P0) edge node[swap] {$f^* \wedge 1$} (P1)
(P1) edge node[swap] {$\otimes$} (P2)
(P2) edge node {$f_*$} (P3)
(P4) edge node {$\otimes$} (P3)
(P0) edge node {$1 \wedge f_*$} (P4);
\end{tikzpicture}
\]
is commutative up to canonically chosen homotopy. 

In particular, for any $x \in K_{\bullet}(X)$ and $y \in K_{\bullet}(Y)$, we have
\begin{equation} \label{ProjForm1}
f_{*}(x \cup f^*y) = f_{*}x \cup y
\end{equation}
in $K_{\bullet}(Y)$.
\end{proposition}

Recall that a morphism $f: X \to Y$ is called a \textit{local complete intersection} (l.c.i) if $f$ can be written as the composition of a closed regular embedding $i: X \to P$ and a smooth morphism $p: P \to Y$.
The class of the relative tangent bundle of $f$ is
$$
N_f: = [i^*T_{P/Y}] - [N_XP] \in K_0(X),
$$
where $N_XP$ is the the normal bundle of $X$ in $P$. $N_{f}$ is independent of the factorization. 

Let 
\begin{equation} \label{eq:excessl.c.i}
\begin{tikzcd}
X' \arrow{r}{f'} \arrow[swap]{d}{g'} & Y' \arrow{d}{g} \\
X \arrow{r}{f}& Y
\end{tikzcd}
\end{equation}
be a Cartesian square where $f$ is a l.c.i morphism. Choose a factorization $f=p\circ i$ as before, and form the Cartesian diagram
\[ \begin{tikzcd}
X' \arrow{r}{i'} \arrow[swap]{d}{g'} & P' \arrow{d} \arrow{r}{p'} \arrow[swap]{d} & Y' \arrow{d}{g} \\
X \arrow{r}{i}& P \arrow{r}{p} & Y.
\end{tikzcd}
\]
Then there is a canonical embedding 
$$
N_{X'}P' \to g'^{*} N_{X}P
$$
of vector bundles on $X'$ \cite[\S 6.1]{Fulton1}.
The excess normal bundle of \eqref{eq:excessl.c.i} is defined by
$$
E: = g'^{*}N_XP / N_{X'}P'.
$$
This definition is independent of the choice of the factorization \cite[Proposition 6.6]{Fulton1}.
Let $E^{\vee}$ be its dual bundle.
\begin{proposition}[Excess intersection formula {\cite[Théorème 3.1]{Thomason2}}, {\cite[Theorem 3.8]{Kock1}}] \label{Excess intersection formula}
Consider the Cartesian diagram \eqref{eq:excessl.c.i} where all the schemes are quasi-projective and $f$ is l.c.i and projective . Then the diagram
\[ \begin{tikzcd}
K_{\bullet}(X') \arrow{r}{f'_{*}}  & K_{\bullet}(Y')  \\
K_{\bullet}(X) \arrow{u}{\lambda_{-1}(E^{\vee}) \cup g'^*} \arrow{r}{f_{*}}& K_{\bullet}(Y) \arrow{u}{g^*}
\end{tikzcd}
\]
commutes, where $\lambda_{-1} E^{\vee}$ is the Euler class $\sum_{i \ge 0}(-1)^i [\wedge^i E^{\vee}] \in K_0(X')$. In other words, for any $x \in K_{\bullet} (X)$,
\begin{equation} \label{ExInterForm1}
g^{*}f_{*} x = f'_{*}(\lambda_{-1}(E^{\vee}) \cup g'^{*} x).
\end{equation}
\end{proposition}

\subsection{Motivic cohomology}

Motivic cohomology is a cohomology theory for algebraic varieties which plays the role of singular cohomology for topological spaces and includes the Chow ring of algebraic cycles as a special case. 
Provisioned by Beilinson, Deligne and constructed by Bloch, Friedlander--Suslin, Voevodsky, etc., this cohomology theory is a key ingredient in Voevodsky's proof of the Milnor conjecture \cite{Voevodsky4} and the motivic Bloch--Kato's conjecture \cite{Voevodsky5}.
Over smooth varieties, all of these constructions are known to be equivalent \cite{Voevodsky1}. 
We choose here Bloch's definition of motivic cohomology via higher Chow groups \cite{Bloch1}, which have an explicitly cycle-theoretical description.

Let $\Delta^r$ be the hyperplane in the affine space $\mathbb{A}_k^{r+1}$ defined by $t_0 + \ldots + t_r = 1$. 
A \textit{face} of $\Delta^r$ is a closed subscheme given by $t_{i_1} = \ldots = t_{i_j} = 0$ for a subset $\{ i_1, \ldots, i_j \} \subset \{0, \ldots, r\}$.

For any $X \in \mathbf{Sch}_k$, let $z^p(X, r)$ be the free abelian group on the irreducible subvarieties of codimension $p$ in $X \times \Delta^r$ which meet all faces \textit{properly}, \emph{i.e.} in the maximal codimension (if not empty).
The assignment $r \mapsto z^p(X, r)$ forms a simplicial abelian group \cite[Introduction]{Bloch1}. The \textit{cycle complex} $z^{p}(X, \bullet)$ is defined to be the complex associated to this simplical group.
\begin{definition}[Bloch \cite{Bloch1}]
The $n$-th \emph{higher Chow group} of algebraic cycles of codimension-$p$, denoted by $\CH^p(X,n)$, is the $n$-th homology group of the complex $z^{p}(X, \bullet)$, i.e.
$$
\CH^p(X,n): = H_n (z^p(X, \bullet)).
$$
\end{definition}

It is straightforward to check that complex $z^{*}(X, \bullet)$ is covariant functorial (with a suitable shift in the grading) for proper morphisms and contravariant functorial for flat morphisms. 

Let $f: X \to Y$ be an arbitrary morphism in $\mathbf{Sm}_k$. Let $z^p_f(Y, n) \subset z^p(Y, n)$ be the subgroup generated by the codimension $p$ subvarieties $Z \subset Y \times \Delta^n$, meeting the faces properly, and 
such that the pull back $X \times Z$ intersects the graph of $f$ properly. Then $z^p_f(Y, \bullet)$ is a chain complex. Using the so-called ``technique of moving cycles'', it is shown that
the inclusion of complexes $z^p_f(Y, \bullet) \subset z^p(Y, \bullet)$ is a quasi-isomorphism. The pull-back by $f$ is defined for algebraic cycles in $z^p_f(Y, \bullet)$. This yields a well-defined homomorphism
$$
f^* \colon \CH^p(Y, n) \to \CH^p(X,n).
$$
Moreover, the assignment $X \to \CH^{*}(X,n)$ is a contravariant functor on the category of smooth, quasi-projective $k$-schemes. For more details, see \cite[Theorem 4.1]{Bloch1} or \cite{Levine1}.

The cycle complexes admit natural associative and commutative external products
$$
\cup_{X,Y}: z^p(X, \bullet) \otimes z^q(Y, \bullet) \to z_{p+q}(X \times Y, \bullet)
$$
in the derived category $D^-(\mathbf{Ab})$ of bounded below complexes of abelian groups. For $X$ smooth over $k$, the pull-back along the diagonal embedding
$
\delta: X \to X \times X
$
induces a natural intersection product in $D^-(\mathbf{Ab})$
$$
\cup_X: = \delta^* \circ \cup_{X,X}: z^p(X, \bullet) \otimes z^q(X, \bullet) \to z^{p+q}(X, \bullet).
$$
These products make $\bigoplus_{p} z^p(X, \bullet)$ an associative and commutative ring in the derived category.
In particular, $\bigoplus_{p,n} \CH^{p}(X,n)$ is a bigraded ring, where the fundamental class $[X] \in \CH^*(X)$ is the  identity element and the product is commutative with respect to the $p$-grading and graded commutative with respect to the $n$-grading 
(i.e., if $x \in \CH^*(X,m)$ and $y \in \CH^*(Y, n)$, then $x \cup y = (-1)^{mn}y \cup x$), see \cite[Corollary 5.7]{Bloch1}.

\begin{remark}
When $n=0$, $z^p(X, 0)$ is the group $z^p(X)$ of codimension $p$ cycles on $X$ and $\CH^{p}(X,0)$ is the quotient of $z^p(X)$ by killing cycles of the form $[Z(0)]- [Z(1)]$, 
where $Z$ is a codimension $p$ cycle on $X \times \mathbb{A}_k^1$ meeting the fibre over $i \in \mathbb{A}_k^1$ properly in $Z(i)$ for $i= 0, \ 1$. 
Hence $\CH^p(X,0)$ is the usual Chow group $\CH^p(X)$ defined by Fulton \cite[1.6]{Fulton1}. 
The restriction of the intersection product '$\cup_X$' to $\bigoplus_p \CH^p(X,0)$ is the usual intersection product on the Chow ring of $X$ \cite[Theorem 5.2 (c)]{Levine1}. 
Note that higher Chow groups considered in \emph{loc.cit.}~are with rational coefficients, however the proof works equally with integral coefficients.
\end{remark}

Similar to algebraic $K$-theory, higher Chow groups satisfy the following properties.
\begin{proposition}[Projection formula]
For any proper map $f: X \to Y$ in $\mathbf{Sm}_k$ and $x \in \CH^*(X, \bullet)$, $y \in \CH^*(Y, \bullet)$, we have
\begin{equation} \label{ProjForm2}
f_{*}(x \cup f^*y) = f_*x \cup y.
\end{equation}
\end{proposition}

\begin{proof}
The proof \cite[Theorem 5.2 (b)]{Levine1} works equally with integral coefficients.
\end{proof}
\begin{proposition} [Excess intersection formula]
Consider a Cartesian diagram \eqref{eq:excessl.c.i} where all varieties are smooth and $g$ is projective. Then for any $x \in \CH^{*}(X, \bullet)$,
$$
g^{*}f_{*} (x) =f'_{*}\left(c_{top}(E) \cup {g'}^{*} x\right)
$$
where $c_{top}(E) \in \CH^{*}(X')$ is the top Chern class of the excess normal bundle $E$.
\end{proposition}
\begin{proof}

This formula is a reformulation of \cite[Proposition 5.17]{Deligse1} under the comparision between motivic cohomology and higher Chow groups \cite[Corollary 2]{Voevodsky1} 
which is compatible with pull-back, proper push-forward \cite[Lecture 19]{MVW} and preserves multiplications \cite[Theorem 3.1]{KY}. 
\end{proof}
When $n=0$, we recover the projection formula and the excess intersection formula for the usual Chow groups.
\subsection{Riemann--Roch theorem}

There are natural relations between $K$-theory and Chow theory (or other cohomology theories) provided by characteristic classes. 
Relying on the work of Gillet \cite{Gillet1}, Bloch defined in \cite[Section 7]{Bloch1} for any $X \in \mathbf{Sm}_k$ and any $n$, a higher Chern character map
\begin{equation} \label{Chern character}
\mathbf{ch}_n^X: K_n(X) \to \bigoplus_p \CH^p(X,n) \otimes \mathbb{Q}
\end{equation}
which generalizes the usual Chern character \cite{FL1}.

\begin{theorem}[{\cite[Theorem 9.1]{Bloch1}, \cite[\S 6]{Levine1}}] \label{Riemann--Roch1}
For any $X \in \mathbf{Sm}_k$ the higher Chern character maps \eqref{Chern character} induce a multiplicative isomorphism
\begin{equation} \label{Chern isomorphism}
\mathbf{ch}: K_{\bullet}(X) \otimes \mathbb{Q} \xrightarrow{\sim} \bigoplus_{p,n} \CH^p(X,n) \otimes \mathbb{Q}.
\end{equation}
\end{theorem}

\begin{remark}
Friedlander--Walker \cite[Theorem 1.10]{FW1} showed that over a field of characteristic $0$, the isomorphism \eqref{Chern isomorphism} can be realized by a map (the Segre map) of infinite loop spaces.

Using exterior powers of vector bundles, Grayson constructed in \cite{Grayson} \textit{$\lambda$-operations} on $K_{\bullet}(X)$ which satisfy the $\lambda$-ring identity \cite[(1.1)]{FL1}. 
As a result, algebraic $K$-theory with rational coefficients splits into a direct sum of weight–graded pieces
$$
K_p(X) \otimes \mathbb{Q} = \bigoplus_q K_p(X)^{(q)}.
$$
In \cite{Levine1}, Levine constructed a morphism $\CH^q(G,p) \otimes \mathbb{Q} \to K_p(X)^{(q)}$ and showed that this is an isomorphism. 
He obtained therefore an isomorphism similar to \eqref{Chern isomorphism} without replying on the higher Chern character.
\end{remark}

Higher Chern characters commute with pull-backs, but do not commute with push-forwards. The lack of commutativity with taking push-forwards is corrected by the Todd classes.
\begin{theorem}[Grothendieck--Riemann--Roch {\cite[Theorem 4.1]{Gillet1}, \cite[Corollary 6.3.2]{Riou1}}] \label{Riemann--Roch2}
Let $f: Y \to X$ be a proper morphism between smooth schemes over $k$, then the following diagram commutes
\[
\begin{tikzcd}
K_{\bullet}(Y) \otimes \mathbb{Q} \arrow[r, "f_{*}"] \arrow[d,"\mathbf{td}(TY)\mathbf{ch}"]
& K_{\bullet}(X) \otimes \mathbb{Q} \arrow[d, "\mathbf{td}(TX) \mathbf{ch}"] \\
\CH^{*}(Y, \bullet) \otimes \mathbb{Q} \arrow[r,  "f_{*}"] & \CH^{*}(X, \bullet) \otimes \mathbb{Q}
\end{tikzcd}
\]
where $\mathbf{td}(TY) \in \CH^{*}(Y) \otimes \mathbb{Q}$ is the Todd class of the tangent bundle of $Y$ and similarly for $\mathbf{td}(TX)$.
\end{theorem}

\section{Equivariant K-theory and motivic cohomology}
We collect here some constructions and results in equivariant geometry. 

\subsection{Basic notions for group actions}
Let $G$ be an algebraic group over $k$. Let $X$ be an algebraic variety over $k$ equipped with an action of $G$, that is, a morphism $\sigma: G\times X\to X$ satisfying the usual axioms. Define the \emph{inertia variety} $I_{G}(X)$ by the following cartesian diagram:
\begin{equation} \label{eq:inertia}
\xymatrix{
I_{G}(X) \ar[r]^{\pi}\ar[d] & X\ar[d]^{\Delta}\\
G\times X \ar[r]^{(\sigma, pr_{2})} & X\times X
}
\end{equation}

\begin{definition}\label{def:FiniteStab}
Let the notation be as before. The action of $G$ on $X$ 
\begin{itemize}
\item is called \emph{proper}, if $(\sigma, pr_{2})$ is proper;
\item has \emph{finite stabilizers}, if $\pi$ is finite; 
\item is \emph{quasi-free}, if $\pi$ is quasi-finite.
\end{itemize}
\end{definition}

\begin{remark}
If $G$ is an affine algebraic group, then a proper action must have finite stabilizers.
If $k$ is of characteristic zero, the quotient stack $[X/G]$ is a Deligne--Mumford stack if and only if $G$ acts quasi-freely. 
By Keel--Mori \cite[Corollary 1.2]{MR1432041}, a coarse moduli space of the quotient stack $[X/G]$ exists only if $G$ acts on $X$ with finite stabilizers.
\end{remark}

\subsection{Equivariant K-theory}\label{subsect:EquiK}

Algebraic $K$-theory has a direct generalization in the equivariant setting.
For any algebraic variety $X$ endowed with an action of a group $G$, the category of $G$-equivariant vector bundles on $X$ is again exact in the sense of \cite{DQ1}. 
Quillen's machinery in \emph{loc.cit.}~produces a connected spectrum $K^{G}(X)$ which is called the equivariant algebraic $K$-theory of $X$ \cite{Thomason1}.
The $i$-th equivariant $K$-group $K^G_{i}(X)$ is by definition the $i$-th homotopy group $\pi_{i}(K^{G}(X))$ and we set
$$
K_{\bullet}^{G}(X): = \bigoplus_{i}K_i^{G}(X).
$$

Similar to the non-equivariant case, the tensor product over $\mathcal{O}_{X}$ makes $K_{\bullet}^{G}(X)$ into a associative and graded commutative ring with identity. 
In the case of points, $K_{0}(G,\operatorname{Spec} (k))$ is the representation ring $R(G)$ of $G$ over $k$. 

If $f:X \to Y$ is a morphism of $G$-schemes, the pull-back of equivariant vector bundles defines a ring homomorphism
$$
f^*: K_{\bullet}^{G}(Y) \to K_{\bullet}^G(X).
$$
This makes $K_{\bullet}^{G}$ a contravariant functor on the category of $G$-schemes.
By pulling-back from the spectrum of the base field, we see that $K_{\bullet}^{G}(X)$ is a graded $R(G)$-algebra, and $f^*$ is a $R(G)$-algebra homomorphism. 

If $f:X \to Y$ is a proper morphism between smooth $G$-schemes over $k$, there is a push-forward map, given by the derived direct image map,
$$ 
f_*: K_{\bullet}^G(X) \to K_{\bullet}^G(Y),$$
which makes $K_{\bullet}^G$ a covariant functor on the category of smooth $G$-schemes with proper morphisms. 
On a smooth $G$-scheme, every $G$-coherent sheaf admits a finite resolution by $G$-equivariant locally free sheaves \cite[Corollary 5.8]{Thomason1} so that
the $K$-theory of $G$-equivariant coherent sheaves is homotopy equivalent to the $K$-theory of $G$-equivariant vector bundles. 
By the projection formula \cite[1.11]{Thomason3}, such $f_*$ is a $R(G)$-module homomorphism.

If $G$ acts freely on $X$, the quotient $X \to X/G$ is a principle $G$-bundle. The category of $G$-equivariant coherent sheaves (resp. $G$-equivariant locally free sheaves) on $X$ is equivalent to the category of coherent sheaves (resp. locally free sheaves) on $X/G$.
This induces a natural isomorphism
$$
K_{\bullet}^G(X) \cong K_{\bullet}(X /G).
$$  

Equivariant algebraic $K$-theory have all the formal properties of algebraic $K$ theory such as the projection formula (\cite[1.11]{Thomason3} or \cite[Proposition 6.5 and Remark 6.6]{VV1}) and the equivariant excess intersection formula 
for finite l.c.i morphisms of schemes satisfying resolution property \cite[Theorem 3.8]{Kock1}.
These properties will be used to prove the associativity of the orbifold product for higher algebraic $K$-theory (Theorem \ref{Associativity3}) along the line of Theorem \ref{Associativity1}.

\subsection{Equivariant higher Chow groups}\label{subsect:EquiChow}

Unlike algebraic $K$-theory, the generalization into the equivariant setting of the motivic cohomology is not the naive one. Roughly speaking, the reason is that there are not enough equivariant cycles on the variety itself for many purposes (for instance, an intersection theory).
Fortunately, this problem has been resolved by Edidin--Graham \cite{EG1}, based on ideas of Totaro on algebraic approximations of classifying spaces \cite{Totaro}. 

\begin{definition}\label{def:EquivHighChow}
Let $X$ be a quasi-projective variety together with a linearizable action of an algebraic group $G$.  For any $n, p\in \mathbb{N}$, the \emph{equivariant higher Chow group} $\CH^{p}_{G}(X, n)$ is defined as
$$
\CH^p_G(X,n): = \CH^p\left( (X \times U)/G,n\right)
$$
where $U$ is a Zariski open subset of some $k$-linear representation $V$ of $G$ such that $V-U$ has codimension at least $p+1$ and $G$ acts freely on $U$. 
\end{definition}
Using Bogomolov's double filtration argument \cite[Definition-Proposition 1]{EG1} and homotopy invariance of Bloch's higher Chow groups on quasi-projective varieties \cite[Theorem 2.1]{Bloch1}, 
we can show easily that this definition is independent of the choice of the pair $(U,V)$. Unlike the ordinary case, for a given index $n$, the group $\CH^p_G(X,n)$ might be non-zero for infinitely many $p$.

\begin{remark}
The assumption on the linearizability of the action of $G$ is necessary to obtain a \textit{quasi-projective} quotient $(X \times U)/G$.
This is not a very strict assumption. Since we will work later with normal and projective schemes, any action of a \textit{connected} algebraic group $G$ is linearizable (Sumihiro's equivariant completion \cite{Sumihiro}). 
\end{remark}

Let $\mathcal{P}$ be one of the following properties of morphisms between schemes: proper, flat, smooth, regular embedding, l.c.i. 
Then equivariant higher Chow groups have the same functorialities as ordinary higher Chow groups for equivariant $\mathcal{P}$ morphisms. 
Indeed, if $f: X \to Y$ is a morphism which satisfies one of these properties, so does the map $f \times \operatorname{id}_{U}: X \times U \to Y \times U$. 
Since $G$ acts freely on $X \times U$ and $Y \times U$, the morphisms $X \times U \to (X \times U)/G$ and $Y \times U \to (Y \times U)/G$ are faithfully flat. 
Moreover, the diagram
\[
\begin{tikzcd}
X \times U \arrow[r, "f \times \operatorname{id}"] \arrow[d]
& Y \times U  \arrow[d] \\
(X \times U)/G\arrow[r, "\overline{f \times \operatorname{id}}"] & (Y \times U)/G
\end{tikzcd}
\]
is Cartesian. Hence by the flat descent \cite{SGA1}, the induced morphism $\overline{f \times \operatorname{id}}: (X \times U)/G \to (Y \times U)/G$ satisfies the same property $\mathcal{P}$.
In particular, the total equivariant higher Chow group has a multiplicative structure and satisfies the projection formula.

If $X \to Y$ is a regular embedding of $G$-schemes over $k$, then the normal bundle $N$ is equipped with a natural action of $G$, and $(N \times U)/G$ is the normal bundle of $(X \times U)/G \to (Y \times U)/G$.
If $E$ is the excess normal bundle of \eqref{eq:excessl.c.i}, then $E$ is a $G$-vector bundle on $X'$, and $(E \times U)/G$ is the excess normal bundle of $(\eqref{eq:excessl.c.i} \times U)/G$. Hence the excess intersection formula holds also for equivariant higher Chow groups. 

When $G$ acts freely on $X$, the projection $X \to X/G$ induces an isomorphism
$$
\CH^p_G(X,n) \cong \CH^p(X/G,n).
$$
If $G$ acts on $X$ with finite stabilizer, the quotient stack $\mathcal{X}:=[X/G]$ is a Deligne--Mumford stack. By the definition in \cite{Kresch}, the Chow group $\CH^{i}(\mathcal{X})$ is nothing else but the equivariant Chow group $\CH^{i}_{G}(X)$. More generally, we set $\CH^{*}(\mathcal{X}, \bullet)=\CH^{*}_{G}(X, \bullet)$. 

\subsection{Equivariant Riemann--Roch}

Define the equivariant higher Chern character map
$$
\mathbf{ch}_n^{G,X}: K^{G}_{n}(X) \to \prod_{p \ge 0} \CH^p_G(X,n) \otimes \mathbb{Q}
$$
whose $p$-th component  $\mathbf{ch}_{n}^{G,X}(p): K_n^{G}(X) \to \CH^p_G(X,n)$ is the composition
$$
K_{n}^{G}(X) \xrightarrow{\pi^*} K_{n}^{G}(X \times U) \xrightarrow{\sim} K_{n}((X \times U)/G) \xrightarrow{\mathbf{ch}_{n}^{(X \times U)/G}(p)} \CH^p((X\times U)/G,n) \otimes \mathbb{Q}
$$
where $U$ is as in Definition \ref{def:EquivHighChow},  $\pi^*$ is the pull-back along the projection $X \times U \to X$ and $\mathbf{ch}_{n}^{(X \times U)/G}(p)$ is the $p$-th component of the higher Chern character $\mathbf{ch}_n^{(X\times U)/G}$ in \eqref{Chern character}.
By Bogomolov's double filtration argument and the homotopy invariance of higher Chow groups, this map is well-defined and independent of the choice of $U$.
These maps resemble to yield a morphism
$$
\mathbf{ch}^{G,X}: K_{\bullet}^G(X)=\bigoplus_{n} K_{n}^{G}(X) \to \bigoplus_n \prod_{p} \CH^p_G(X,n) \otimes \mathbb{Q}.
$$
which is also called the equivariant higher Chern character. 
As direct consequences of the corresponding properties in the non-equivariant setting (see Theorem \ref{Riemann--Roch1} or \cite{Bloch1}), this map is a ring homomorphism and commutes with pull-backs. 
If $G$ acts freely on $X$, we can identify $\mathbf{ch}_n^{G,X}$ with $\mathbf{ch}_n^{X/G}$.

Krishna generalized Theorem \ref{Riemann--Roch2} to the equivariant setting:
\begin{theorem} [Equivariant Riemann--Roch {\cite[Theorem 1.4]{Krishna1}}]
Let $f: Y \to X$ be a proper equivariant morphism between smooth quasi-projective $G$-schemes over $k$, then the following diagram commutes
\[
\begin{tikzcd}
K_{\bullet}^{G}(Y) \otimes \mathbb{Q} \arrow[r, "f_{*}"] \arrow[d,"\mathbf{td}^G(TY)\mathbf{ch}^{G,Y}"]
& K_{\bullet}^{G}(X) \otimes \mathbb{Q} \arrow[d, "\mathbf{td}^G(TX) \mathbf{ch}^{G,X}"] \\
\CH^{*}_{G}(Y, \bullet) \otimes \mathbb{Q} \arrow[r,  "f_{*}"] & \CH^{*}_{G}(X, \bullet) \otimes \mathbb{Q}
\end{tikzcd}
\]
where $\mathbf{td}^G(TY) \in \CH^{*}_{G}(Y) \otimes \mathbb{Q}$ is the equivariant Todd class of the tangent bundle $TY$ of $Y$.
\end{theorem}

To generalize Theorem \ref{Riemann--Roch1}, we need to recall the completion construction. Let $I_G \subset R(G)=K_{\bullet}^{G}(pt)$ be the augmentation ideal, that is, the kernel of $\dim: R(G)\to \mathbb{Z}$. For any $m \in \mathbb{N}$, we have
$$
\mathbf{ch}^{G,pt}(I_G^m) \subset \prod_{p = m}^{\infty} \CH^p_G(pt) \otimes \mathbb{Q}.
$$
Hence $\mathbf{ch}^{G,X} (I_G^m K^G_\bullet(X)) \subset \prod _{p = m}^{\infty} \CH^p_G(X, \bullet)$ by multiplicative property of the higher Chern character. 
This means that the Chern character map induces a morphism
$$
K^G_\bullet(X) / I_G^mK^G_\bullet(X) \to \prod_{p = 0}^m \CH_G^p(X, \bullet) \otimes \mathbb{Q}.
$$
Taking the projective limit over $m \in \mathbb{N}$ gives rise to a morphism
$$
K^G_\bullet(X)^{\wedge} \to \prod_{p=0}^{\infty} \CH^p_G (X, \bullet) \otimes \mathbb{Q}
$$
where $K^G_\bullet(X)^{\wedge}=\varprojlim_{m} K^G_\bullet(X) / I_G^mK^G_\bullet(X)$ is the completion of $K^G_\bullet(X)$ with respect to the $I_G$-adic topology.

When $G$ acts on $X$ with finite stabilizers (Definition \ref{def:FiniteStab}), the group $\CH^p_G(X,n) \otimes \mathbb{Q}$ is generated by invariant codimension $p$ cycles on $X \times \Delta^n$, hence it vanishes for $p$ sufficiently large \cite[Proposition 7.2]{Krishna1}.
So we can identify the infinite direct product $\prod_{p \ge 0}\CH^{p}_{G}(X,n) \otimes \mathbb{Q}$
with the direct sum $\bigoplus_{p \ge 0} \CH^{p}_{G}(X,n)$. 

Now we can state Krishna's generalization of Theorem \ref{Riemann--Roch1} to the equivariant setting:
\begin{theorem} [Atiyah--Segal's completion {\cite[Theorem 4.6]{Krishna1}}] \label{Completion}
If $G$ acts on $X$ with finite stabilizers, the equivariant higher Chern character factors through an isomorphism
$$
\mathbf{ch}^{G,X}: K_{\bullet}^G(X)^{\wedge} \otimes \mathbb{Q} \to  \bigoplus_{p,n} \CH^p_G(X,n) \otimes \mathbb{Q}.
$$
\end{theorem}

\section{Orbifold theories: global quotient by a finite group}\label{sect:FiniteGroupQuotient}
In this section, we construct the stringy and orbifold K-theory and motivic cohomology for Deligne--Mumford stacks in the more restricted case of a global quotient of a smooth projective variety by a \emph{finite} group action. We follow \cite{JKK1} when developing this theory.
From now on, the base field $k$ is the field of complex numbers $\mathbb{C}$. Higher $K$-groups and higher Chow groups are always with rational coefficients.

\subsection{Higher inertia varieties}\label{subsect:HigherInertia}
Let $X$ be a smooth projective variety over $k$ endowed with a left action of a finite group $G$. For any $g\in G$, let $X^{g}$ be the fixed locus of $g$. 
More generally, for any subgroup $H \subset G$, denote the fixed locus of $H$ in $X$ by $X^H$, i.e., $X^{H}$ is the biggest closed subscheme, with the reduced scheme structure, of $X$ on which $H$ acts trivially. 
By our assumption, $X^{H}$ is smooth over $k$ \cite[Proposition 3.4]{BE1}. If $\mathbf{g} = (g_1, \ldots, g_n) \in G^n$, we will write $X^{\mathbf{g}}$, or sometime $X^{g_1, \ldots, g_n}$, for the fixed locus $X^{\langle g_1, \ldots, g_n \rangle}$.

Recall from \eqref{eq:inertia} that the \textit{inertia variety} of $X$ (with respect to $G$) is defined to be the disjoint union
$$
I_G(X): = \coprod_{g \in G} X^{g} \subset G \times X.
$$
We equip $I_G(X)$ with the $G$-action given by $h.(g,x)= (hgh^{-1}, hx)$.

Similarly, for any $n \in \mathbb{N}_{>0}$ the \textit{$n$-th inertia variety} $I^{n}_G(X)$ of $X$ (with respect to $G$) is
$$
I^{n}_G(X): = \coprod_{(g_1, \ldots, g_n) \in G^n} X^{( g_1, \ldots, g_n)} \subset G^n \times X,
$$
equipped with the action of $G$ given by
\begin{equation} \label{Action} 
h.(g_1, \ldots, g_n, x)  = (hg_1h^{-1}, \ldots, h g_nh^{-1}, hx).
\end{equation}
For convenience, we set $I^0_G(X): = X$. In the case of \textit{abelian groups}, $I^{n+1}_G(X)$ is obviously the inertia variety of $I^n_G(X)$. 

There are \textit{face maps}
\begin{equation} \label{eq:face}
\begin{split}
f^n_i \colon I^{n}_G(X) & \to I^{n-1}_G(X) \\
(g_1, \ldots, g_n, x) & \mapsto (g_1, \ldots, g_{i-1} g_i, \ldots, g_n, x)
\end{split}
\end{equation}
for $1 \le i \le n$, and \textit{degeneracy maps}
\begin{align*}
d^n_i \colon I^{n}_G(X) & \to I^{n+1}_G(X) \\
(g_1, \ldots, g_n, x) & \mapsto (g_1, \ldots, g_{i-1}, 1_{G}, g_i, \ldots, g_n, x)
\end{align*}
for $1 \le i \le n+1$. It is straightforward to check that $I^{\bullet}_G(X)$ together with these maps form a simplicial scheme.
In the case of a point with the trivial action, $I^{\bullet}_G(\operatorname{pt})$ is the simplicial group defining the classifying space of $G$.

For $1 \le i \le n$, there are also \textit{evaluation maps}
\begin{equation} \label{eq:ev}
\begin{split} 
e^n_i \colon I^{n}_G(X) & \to I^{n-1}_G(X) \\
(g_1, \ldots, g_n, x) & \mapsto (g_1, \ldots, \hat{g_i}, \ldots, g_n, x)
\end{split}
\end{equation}
and the involution map
\begin{align*}
\sigma: I^n_G(X) & \to I^n_G(X)\\
(g_1, \ldots, g_n ,x) & \mapsto (g_1^{-1}, \ldots, g_n^{-1}, x)
\end{align*}
All these maps defined above are obviously equivariant via the action \eqref{Action}.

For our purpose, we will not use the simplicial structure of $I^{\bullet}_G(X)$ but only work with the inertia varitey $I_{G}(X)$ and the double inertia variety $I^2_G(X)$.
We will see later that the maps $e^2_1, \, e^2_2, \, f^2_2, \,\sigma$ together with the obstruction bundle on $I^2_G(X)$ are enough to obtain new interesting invariants on $I_G(X)$. 
Nevertheless, there is no doubt that understanding higher inertia varieties together with their simplical structure will gives us a more complete picture about stringy (and later, orbifold) theories.     
\begin{remark} [Convention]
We always understand that $X^g$ is $\{g \} \times X^{g} \subset I_G(X)$. We write $X^{g,h}$ for $\{(g,h)\} \times X^{\langle g,h\rangle}$ in $I^2_G(X)$. 
These conventions keep track of components of the (double) inertia variety.
\end{remark}
We set $e_1 := e^2_2, \, e_2 := e^2_1$ and $\mu : = f^2_2$. 
With this convention, the evaluation maps $e_i: I^{2}_G(X) \to I_G(X)$ are the disjoint union of the inclusions $X^{\langle g_1, g_2\rangle} \hookrightarrow X^{g_i}$ for $i = 1, \, 2$, and the \textit{multiplication}
$\mu \colon I^2_G(X) \to I_G(X)$ is the disjoint union of the inclusions $X^{\langle g, h\rangle} \hookrightarrow X^{gh}$. 
It is shown in \cite{EJK1} that by passing to the quotients by $G$, the three maps $e_1, e_2, \sigma \circ \mu$ give the evaluation maps 
$$
e_1, e_2, e_3 \colon \mathcal{K}_{0,3}(\mathcal{X},0) \rightarrow \mathcal{I}_{\mathcal{X}}  
$$
from the stack $\mathcal{K}_{0,3}(\mathcal{X},0)$ of three pointed genus-0 degree-0 twisted stable maps to $\mathcal{X}$, to the inertia stack $\mathcal{I}_{\mathcal{X}}$ of $\mathcal{X}$ studied in \cite{AGV1}.

\subsection{Stringy K-theory}

\begin{definition}
The \textit{stringy K-theory} $K(X,G)$ of $X$, as a graded vector space, is defined to be the rational $K$-theory of its inertia variety, i.e.,
$$
K_{\bullet}(X,G): = K_{\bullet}(I_G(X)) = \prod_{g \in G} K_{\bullet}(X^g).
$$
The groups $G$ acts on $K_{\bullet}(X,G)$ via their actions on $I_G(X)$.
\end{definition}

With the usual multiplication given by tensor product of vector bundles, it is well-known that there is an isomorphism of $R(G)$-algebras
$$
K_{\bullet}^{G}(X)\otimes \mathbb{C} \cong \left(K_{\bullet}(I_G(X)) \otimes \mathbb{C}\right)^G,
$$
see \cite[Theorem 5.4]{VV1} or \cite[Theorem 1]{AV1}.
\begin{definition} \label{Logarithmic1}
Define $\Im \in K_0(I_G(X))$ (the rational $K_0$-group) to be such that for any $g \in G$, its restriction $\Im_g$ in $K_0(X^g)$ is given by
$$
\Im_g:= \Im|_{X^{g}}: = \sum_{k} \ \alpha_k [W_{g,k}],
$$
where $0 \le \alpha_k <1$ are rational numbers such that $\operatorname{exp}(2\pi i \alpha_k)$ are the eigenvalues of $g$ on the normal bundle $N_{X^g} X$ and 
$W_{g,k}$ are the corresponding eigenbundles.
\end{definition}

It is straightforward to check that
\begin{equation} \label{Equation6}
\Im_g + \sigma^{*} \Im_{g^{-1}} = [N_{X^g} X]
\end{equation}
in $K_0(X^g)$.

\begin{definition}[Age]\label{def:age}
Notation is as before, the \emph{age function}, denoted by $\age(g)$, is the locally constant function on $X^{g}$ defined by $\rk(\Im_{g})$.
\end{definition}

\begin{definition}[Obstruction bundle] \label{def:obstruction}
The obstruction bundle class $\mathcal{R}$ is the element in $K_0(I^2_G(X))$ whose restriction to $ K_0 (X^{\mathbf{g}})$ is given by
\begin{align*}
\mathcal{R} (g_1, g_2): & = \Im_{g_1}|_{X^{\mathbf{g}}} + \Im_{g_2}|_{X^{\mathbf{g}}} + \Im_{(g_1g_2)^{-1}}|_{X^{\mathbf{g}}} - [N_{X^{\mathbf{g}}}X] \\
& = e_1^*(\Im_{g_1}) + e_2^*(\Im_{g_2}) + (\sigma \circ \mu)^*(\Im_{(g_1g_2)^{-1}}) - [N_{X^{\mathbf{g}}}X]
\end{align*}
for any $\mathbf{g}=(g_1, g_2) \in G^2$.
\end{definition}

Jarvis--Kaufmann--Kimura have shown \cite[Theorem 8.3]{JKK1} that $\mathcal{R}(\mathbf{g})$ is represented by a vector bundle on $X^{\mathbf{g}}$ 
which is the obstruction bundle of Fantechi--Göttsche for stringy cohomology \cite{FG1}. In particular, $\mathcal{R}(\mathbf{g})$ is a positive element in $ K_0 (X^{\mathbf{g}})$. 

\begin{definition}[Stringy product]
Given $(g_1, g_2) \in G^2$. For any $x \in K_{\bullet}(X^{g_1})$ and $y \in K_{\bullet}(X^{g_2})$, we define the stringy product of $x$ and $y$ to be
\begin{equation} \label{Stringy product K theory} 
x \star y: = \mu_{*}\left(e^{*}_{1}x \cup e^{*}_{2}y \cup \lambda_{-1}\left(\mathcal{R}(g_1, g_2)^{\vee}\right)\right) \in K_{\bullet}(X^{g_1g_2})
\end{equation}
which is extended linearly to a product on $K_{\bullet}(X,G)$.
\end{definition}

\begin{theorem} \label{Associativity1}
The product \eqref{Stringy product K theory} is associative (but may not be commutative).
\end{theorem}
In order to prove this theorem, we will need the following lemma. See \S \ref{subsect:Ktheory} for the definition of excess normal bundle. 
\begin{lemma} [{\cite[Lemma 5.2]{JKK1}}]
Let $\mathbf{g}: = (g_1, g_2, g_3) \in G^3$. 
Let $E_{1,2}$ be the excess normal bundle of 
\begin{equation} \label{eq:excess1} 
\begin{tikzcd}
X^{\mathbf{g}} \arrow[r] \arrow[d]
& X^{g_1, g_2} \arrow[d] \\
X^{g_1g_2,g_3} \arrow[r]
& X^{g_1g_2}
\end{tikzcd}
\end{equation}
and $E_{2,3}$ the excess normal bundle of 
\begin{equation} \label{eq:excess2} 
\begin{tikzcd}
X^{\mathbf{g}} \arrow[r] \arrow[d]
& X^{g_2, g_3} \arrow[d] \\
X^{g_1,g_2g_3} \arrow[r]
& X^{g_2g_3},
\end{tikzcd}
\end{equation} where all the morphisms are the natural inclusions. Then the following equation holds in $K_0(X^{\mathbf{g}})$
\begin{equation} \label{Equation1}
 \mathcal{R}(g_1, g_2)|_{X^{\mathbf{g}}} + \mathcal{R}(g_1 g_2, g_3))|_{X^{\mathbf{g}}} + [E_{1, 2}] \
= \ \mathcal{R}(g_1, g_2 g_3)|_{X^{\mathbf{g}}} + \mathcal{R}(g_2, g_3)|_{X^{\mathbf{g}}} + [E_{2, 3}].
\end{equation}
More precisely, they are equal to
$
\sum_{i=1}^{3} \Im_{g_i} |_{X^{\mathbf{g}}} + \Im_{(g_1g_2g_3)^{-1}} |_{X^{\mathbf{g}}} - N_{X^{\mathbf{g}}}X.
$
\end{lemma}

\begin{proof}[Proof of Theorem \ref{Associativity1}]
We follow the proof of \cite[Lemma 5.4]{JKK1} closely.
The multiplicativity of $\lambda_{-1}$ applies to \eqref{Equation1} yields
\begin{equation} \label{Equation2}
\begin{split}
 & \lambda_{-1} (\mathcal{R}(g_1, g_2)^\vee)|_{X^{\mathbf{g}}} \cup \lambda_{-1}(\mathcal{R}(g_1 g_2, g_3)^\vee))|_{X^{\mathbf{g}}} \cup \lambda_{-1} (E_{1, 2})^\vee \\
= \ & \lambda_{-1}(\mathcal{R}(g_1, g_2 g_3)^\vee)|_{X^{\mathbf{g}}} \cup \lambda_{-1}(\mathcal{R}(g_2, g_3)^\vee)|_{X^{\mathbf{g}}} \cup \lambda_{-1} (E_{2, 3})^\vee.
\end{split}
\end{equation}

Let $g_4: = g_1 g_2 g_3$ and consider the following diagram
\begin{equation}
\xymatrix{
&&X^{\mathbf g} \ar[dl]_{e} \ar[dr]^{f}&&\\
&X^{g_{1}, g_{2}} \ar[dl]_{e_1} \ar[d]^{e_2} \ar[dr]^{\mu}& &X^{g_1g_2, g_{3}}\ar[dl]_{e_1} \ar[d]^{e_2} \ar[dr]^{\mu}&\\
X^{g_{1}}&X^{g_{2}}&X^{g_1g_2}&X^{g_{3}}&X^{g_{4}}
}
\end{equation}
where $e$ is the evaluation map $e^3_3$ \eqref{eq:ev} and $f$ is the face map $f^3_2$ \eqref{eq:face}. The middle rhombus is the diagram \eqref{eq:excess1}. The natural embeddings $j_i: X^{\mathbf{g}} \to X^{g_i}$ factor as
\begin{align*}
& j_1 = e_{1} \circ e \ \ \ \ \ \ \ \ \ \ \ \   j_2 = e_{2} \circ e \\
& j_3 = e_{2} \circ f \ \ \ \ \ \ \ \ \ \ \ \   j_4 = \mu \circ f.
\end{align*}

For any $x \in K_{\bullet}(X^{g_1})$, $y \in K_{\bullet}(X^{g_2})$, $z \in K_{\bullet}(X^{g_3})$, we have
\begin{equation}\label{Equation3}
\begin{split} 
& (x \star y) \star z \\ 
:= & \ \mu_* \big\{ e_{1}^{*}\left[\mu_*  \left( e^{*}_{1} x \cup e^{*}_{2} y \cup \lambda_{-1}(\mathcal{R}(g_1, g_2)^{\vee}) \right)\right] \cup e^{*}_{2}z \cup \lambda_{-1}(\mathcal{R}(g_1g_2, g_3)^{\vee})\big \} \\
= & \ \mu_* \big\{ f_{*} \left[ e^{*} \left( e^{*}_{1} x \cup e^{*}_{2}y \cup \lambda_{-1}(\mathcal{R}(g_1, g_2)\right) \cup \lambda_{-1}(E^\vee_{1, 2})\right]  \cup e^{*}_{2}z \cup \lambda_{-1}(\mathcal{R}(g_1g_2, g_3)^{\vee}) \big\} \\ 
= & \ \mu_* \big\{ f_{*} \left[ e^{*}  e^{*}_{1} x \cup e^{*} e^{*}_{2}y \cup e^{*} \lambda_{-1}(\mathcal{R}(g_1, g_2)^{\vee})  \cup \lambda_{-1}(E^\vee_{1, 2}) \right]   \\ 
  & \ \ \ \ \ \ \ \ \ \ \ \ \ \ \ \ \ \ \ \ \ \ \ \ \ \ \ \ \ \ \ \ \ \ \ \ \ \  \ \ \ \ \ \ \ \ \ \ \ \ \ \ \ \ \ \ \ \ \ \ \ \ \ \ \ \ \ \cup e^{*}_{2}z \cup \lambda_{-1}(\mathcal{R}(g_1g_2, g_3)^\vee) \big\} \\
= & \ \mu_* \big\{ f_{*} \big[ e^{*} e^{*}_{1} x \cup e^{*} e^{*}_{2}y \cup e^{*} \lambda_{-1}(\mathcal{R}(g_1, g_2) \cup \lambda_{-1}(E^\vee_{1, 2})    \\ 
  & \ \ \ \ \ \ \ \ \ \ \ \ \ \ \ \ \ \ \ \ \ \ \ \ \ \ \ \ \ \ \ \ \ \ \ \ \ \ \ \ \ \ \ \ \ \ \ \ \ \ \ \ \ \ \ \ \ \ \ \cup f^{*} e^{*}_{2}z \cup f^{*} \lambda_{-1}(\mathcal{R}(g_1g_2, g_3)^{\vee}) \big] \big\}  \\
= & \ i_{4*} \big(j^{*}_1 x \cup j^{*}_2 y \cup e^{*}(\lambda_{-1} \mathcal{R}(g_1, g_2) \cup \lambda_{-1}(E^{\vee}_{1, 2}) \cup j_3^{*} z \cup 
f^{*}(\lambda_{-1} \mathcal{R}(g_1g_2, g_3)^{\vee})\big) \\ 
= & \ j_{4*} \big(j^{*}_1 x \cup j^{*}_2 y  \cup j_3^{*} z \cup \lambda_{-1} \mathcal{R}(g_1, g_2)^{\vee}|_{X^{\mathbf{g}}}  
\cup \lambda_{-1}  \mathcal{R}(g_1g_2, g_3)^{\vee}|_{X^{\mathbf{g}}} \cup \lambda_{-1}E^{\vee}_{1, 2}\big),
\end{split}
\end{equation}
where the second equility follows from the excess intersection formula (Proposition \ref{Excess intersection formula}), 
the fourth equality follows from the projection formula (Proposition \ref{Projection formula}).

Using a similar argument, we have
\begin{equation} \label{Equation4}
\begin{split}
x \star (y \star z) 
= j_{4*} & (j^{*}_1 x \cup j^{*}_2 y  \cup j_3^{*} z \cup \lambda_{-1} \mathcal{R}(g_1, g_2g_3)^{\vee}|_{X^{\mathbf{g}}}  
\cup \lambda_{-1}  \mathcal{R}(g_2, g_3)^{\vee}|_{X^{\mathbf{g}}} \cup \lambda_{-1}E^{\vee}_{2, 3}).
\end{split}
\end{equation}
By \eqref{Equation2}, the two expressions \eqref{Equation3} and \eqref{Equation4} are equal.
\end{proof}

\subsection{Stringy higher Chow groups}
\begin{definition}
We define the \textit{stringy higher Chow group} $\CH^*(X,G, \bullet)$ of X, as a bigraded vector space, to be the rational higher Chow group of the inertia variety, with codimension degree shifted by the age function (Definition \ref{def:age}), i.e.,
$$
\CH^i(X,G, \bullet): = \CH^{i-\age}(I_G(X), \bullet) = \prod_{g \in G} \CH^{i-\age(g)}(X^g, \bullet).
$$
\end{definition}
Note that there are isomorphisms (see \cite[Theorem 3]{EG1}):
$$
\CH^p_G(X,n) \otimes \mathbb{Q} \cong \CH^p(X/G,n) \otimes \mathbb{Q} \cong \left(\CH^p(X, n) \otimes \mathbb{Q}\right)^G.
$$

\begin{definition}[Stringy product]
Given $\mathbf{g} = (g_1, g_2) \in G^2$. For any $x \in \CH^{i-\age(g_{1})}(X^{g_1}, \bullet)$ and $y \in \CH^{j-\age(g_{2})}(X^{g_2}, \bullet)$, we define the stringy product of $x$ and $y$ in $\CH^{i+j-\age(g_{1}g_{2})}(X^{g_{1}g_{2}}, \bullet)$ to be
\begin{equation} \label{Stringy product motivic cohomology} 
x \star y: = \mu_* (e^{*}_{1} x \cup e^{*}_{2} y \cup c_{top}(\mathcal{R}(\mathbf{g}))).
\end{equation}
We extend linearly this product to the whole $\CH^*(X,G, \bullet)$.
\end{definition}

The multiplicativity of the top Chern character and the equality \eqref{Equation1} give
\begin{equation} \label{Equation5}
\begin{split}
& c_{top} (\mathcal{R}(g_1, g_2))|_{X^{\mathbf{g}}} \cup c_{top}(\mathcal{R}(g_1 g_2, g_3)))|_{X^{\mathbf{g}}} \cup c_{top} (E_{1, 2}) \\
=  & \, c_{top}(\mathcal{R}(g_1, g_2 g_3))|_{X^{\mathbf{g}}} \cup c_{top}(\mathcal{R}(g_2, g_3))|_{X^{\mathbf{g}}} \cup c_{top} (E_{2, 3}).
\end{split}
\end{equation}
A similar argument as in the proof of Theorem \eqref{Associativity1} (\emph{cf.} \cite[Lemma 5.4]{JKK1}) shows that

\begin{theorem} \label{Associativity2}
The stringy product for stringy motivic cohomology \eqref{Stringy product motivic cohomology} is associative.
\end{theorem}

Following \cite[(6.1)]{JKK1} we introduce the following:
\begin{definition}[Stringy Chern character]
The \textit{stringy Chern character} $\mathcal{C}\mathbf{h}: K_{\bullet}(X,G) \to \CH^*(X,G, \bullet)$ to be
\begin{equation} \label{Stringy Chern character}
\mathcal{C}\mathbf{h}(x_g): = \mathbf{ch}(x_g) \cup \mathbf{td}^{-1}(\Im_g)
\end{equation}
for all $g \in G$ and $x_g \in K_{\bullet}(X^g)$, where $\mathbf{td}$ is the usual Todd class and $\mathbf{ch}$ is the higher Chern character map in Theorem \ref{Riemann--Roch1}.
\end{definition}

\begin{theorem} \label{thm:stringChern}
The stringy Chern character is a multiplicative homomorphism, with respect to the stringy products.
\end{theorem}
\begin{proof}
The proof is along the same line of \cite[Theorem 6.1]{JKK1}, with techniques on K-theory/Chow-theory replaced by their higher analogues.
Recall that if $E$ is a vector bundle on $X$, then
\begin{equation} \label{Todd-Chern}
\mathbf{td}([E]) \mathbf{ch}(\lambda_{-1}[E^\vee]) = c_{top}([E])
\end{equation}
in $\CH^{*}(X)$.

Let $\mathbf{g}=(g_1, g_2) \in G^2$ and $x_{i} \in K_{\bullet}(X^{g_i})$ for $i= 1, 2$. Set $g = g_1g_2$ and $\mathcal{R} = \mathcal{R}(g_1, g_2)$. We have
\begin{equation*}
\begin{split}
\mathcal{C}\mathbf{h} (x_{1} \star x_{2})
& = \ \mathbf{ch}(x_{1} \star x_{2}) \mathbf{td}^{-1}(\Im_{g}) \\
& = \ \mathbf{ch} \left[\mu_* \left(e^*_{1}x_{1} e^*_{2}x_{2} \lambda_{-1}(\mathcal{R}^\vee)\right) \right] \mathbf{td}^{-1}(\Im_{g})  \\
& = \ \mu_* \left[\mathbf{ch} \left(e^*_{1}x_{1} e^*_{2} x_{2} \lambda_{-1}(\mathcal{R}^\vee) \mathbf{td}(TX^{\mathbf{g}}) \right) \right] \mathbf{td}^{-1}(TX^{g}) \mathbf{td}^{-1}(\Im_{g})  \\
& = \ \mu_* \left[e^*_{1}\mathbf{ch}(x_{1}) e^*_{2}\mathbf{ch}(x_{2}) \mathbf{ch}(\lambda_{-1}(\mathcal{R}^\vee)) \mathbf{td}(TX^{\mathbf{g}})) \right] \mathbf{td}^{-1}(TX^{g}) \mathbf{td}^{-1}(\Im_{g})  \\
& = \ \mu_* \left[e^*_{1}\mathbf{ch}(x_{1}) e^*_{2}\mathbf{ch}(x_{2}) c_{top}(\mathcal{R})\mathbf{td}^{-1}(\mathcal{R}) \mathbf{td}(TX^{\mathbf{g}}] \right] \mathbf{td}^{-1}(TX^{g}) \mathbf{td}^{-1}(\Im_{g})  \\
& = \ \mu_* \left[e^*_{1}\mathbf{ch}(x_{1}) e^*_{2}\mathbf{ch}(x_{2}) c_{top}(\mathcal{R}) \mathbf{td}(TX^{\mathbf{g}} - \mathcal{R}) \right] \mathbf{td}(- TX^{g} - \Im_{g})  \\
& = \ \mu_* \left[e^*_{1}\mathbf{ch}(x_{1}) e^*_{2}\mathbf{ch}(x_{2}) c_{top}(\mathcal{R}) \mathbf{td}(TX^{\mathbf{g}} - \mathcal{R}) \mu^* \mathbf{td}(- TX^{g} - \Im_{g}) \right]  \\
& = \ \mu_* \left[e^*_{1}\mathbf{ch}(x_{1}) e^*_{2}\mathbf{ch}(x_{2}) c_{top}(\mathcal{R}) \mathbf{td}(TX^{\mathbf{g}} - \mathcal{R} - \mu^* TX^{g} - \mu^* \Im_{g}) \right]  \\
\end{split}
\end{equation*}
where the first two equalities follow from definition, the third follows from the Riemann--Roch Theorem \ref{Riemann--Roch2}, the fourth from the fact that the higher Chern character respects pull-backs and multiplications,
the fifth follows from \eqref{Todd-Chern}, the seventh is the projection formula, the sixth and the eighth follow from multiplicativity of $\mathbf{td}$.

We also have
\begin{equation*}
\begin{split}
\mathcal{C}\mathbf{h}(x_{1}) \star \mathcal{C}\mathbf{h}(x_{2})
= & \left(\mathbf{ch}(x_{1}) \mathbf{td}^{-1} (\Im_{g_1}) \right) \star \left(\mathbf{ch}(x_{2}) \mathbf{td}^{-1} (\Im_{g_2})\right) \\
= & \mu_* \left[ e^*_{1} \left(\mathbf{ch}(x_{1}) \mathbf{td}^{-1}(\Im_{g_1})\right) e^*_{g_2} \left(\mathbf{ch}(x_{2}) \mathbf{td}^{-1}(\Im_{g_2})\right) c_{top}(\mathcal{R}) \right] \\
= & \mu_* \left[ e^*_{1}\mathbf{ch}(x_{1}) \mathbf{td}^{-1}(e^*_{1}\Im_{g_1}) e^*_{g_2}\mathbf{ch}(x_{2}) \mathbf{td}^{-1}(e^*_{2}\Im_{g_2}) c_{top}(\mathcal{R}) \right] \\
= & \mu_* \left[ e^*_{1}\mathbf{ch}(x_{1}) e^*_{2}\mathbf{ch}(x_{2}) c_{top}(\mathcal{R}) \mathbf{td}(- e^*_{1}\Im_{g_1} - e^*_{2}\Im_{g_2}) \right]  \\
\end{split}
\end{equation*}
where the first two equalities are definitions, the third holds because pull-backs respect multiplication, the fourth follows from the multiplicativity of $\mathbf{td}$.
Hence, it is sufficient to prove that 
\begin{equation} \label{eq:compare}
TX^{\mathbf{g}} - \mathcal{R} - \mu^* TX^{g} - \mu^* \Im_{g} = - e^*_{1}\Im_{g_1} - e^*_{2}\Im_{g_2}
\end{equation}
in $K_0(X^{g_1,g_2})$. Indeed, 
\begin{equation*}
\begin{split}
TX^{\mathbf{g}} - \mathcal{R} - \mu^* TX^{g} - \mu^* \Im_{g}
= & TX^{\mathbf{g}} - \mathcal{R} - TX^{g}|_{X^{\mathbf{g}}} - \Im_{g}|_{X^{\mathbf{g}}} \\
= & (TX^{\mathbf{g}} + N_{X^{\mathbf{g}}}X) - TX^{g}|_{X^{\mathbf{g}}} - \Im_{g}|_{X^{\mathbf{g}}} -  \mathcal{R} - N_{X^{\mathbf{g}}}X\\
= & TX|_{X^{\mathbf{g}}} - TX^{g}|_{X^{\mathbf{g}}}  - \Im_{g}|_{X^{\mathbf{g}}} - (\mathcal{R} + N_{X^{\mathbf{g}}}X) \\
= & N_{X^{g}}X|_{X^{\mathbf{g}}} - \Im_{g}|_{X^{\mathbf{g}}} -  (\mathcal{R} + N_{X^{\mathbf{g}}}X) \\
= & \sigma^* \Im_{g^{-1}}|_{X^{\mathbf{g}}} - (\Im_{g_1}|_{X^{\mathbf{g}}} + \Im_{g_2}|_{X^{\mathbf{g}}} + \Im_{g^{-1}}|_{X^{\mathbf{g}}}) \\
= & \mu^* \sigma^* \Im_{g^{-1}} - e^*_{g_1}\Im_{g_1} - e^*_{g_2}\Im_{g_2} - (\sigma \circ \mu)^* \Im_{g^{-1}}\\
= & - e^*_{1}\Im_{g_1} - e^*_{2}\Im_{g_2}
\end{split}
\end{equation*}
where the first equality is definition, the third and the fourth are the natural relation of tangent bundle and normal bundle in the Grothendieck group, the fifth follows from \eqref{Equation6}.
\end{proof}

\begin{lemma}
The stringy products on $K_{\bullet}(I_G(X))$, $\CH^*(I_G(X), \bullet)$ are compatible with the $G$-actions. 
\end{lemma}
\begin{proof}
It is easy to check that $\mathcal{R}(g^{-1} g_1 g, g^{-1} g_2 g) = g^*\mathcal{R}(g_1, g_2)$ for any $g, \, g_1,\, g_2 \in G$. Hence 
\begin{equation*}
\lambda_{-1} \left(\mathcal{R}(g^{-1} g_1 g, g^{-1} g_2 g)^\vee \right) = g^*\lambda_{-1} \left(\mathcal{R}(g_1, g_2)^\vee\right).
\end{equation*}
The identity $g \circ e_i = e_i \circ g$ implies that $g^* \circ e_i^* = e_i^* \circ g^*$ for $i=1,\, 2$. The diagram
\[
\begin{tikzcd}
I^2_G(X) \arrow[r, "\mu"] \arrow[d, "g"]
& I_G(X) \arrow[d, "g"] \\
I^2_G(X) \arrow[r, "\mu"]
& I_G(X)
\end{tikzcd}
\]
is Cartesian. It follows that $g^* \mu_* = \mu_* g^*$ on $K$-groups by \cite[Proposition 3.18]{TT1}. So we have
\begin{align*}
(g^*x) \star (g^*y) & = \mu_{*}\left[e^{*}_{1} g^*x \cup e^{*}_{2}g^* y \cup \lambda_{-1}(\mathcal{R}(g^{-1}g_1g, g^{-1}g_2g)^{\vee})\right] \\
& =  \mu_{*}\left[g^*e^{*}_{1} x \cup g^* e^{*}_{2} y \cup g^* \lambda_{-1}(\mathcal{R}(g_1, g_2)^{\vee})\right] \\
& =  \mu_{*} \left[ g^* \left(e^{*}_{1} x \cup e^{*}_{2} y \cup \lambda_{-1}(\mathcal{R}(g_1, g_2)^{\vee}) \right) \right] \\
& =  g^* \left[\mu_{*} \left(e^{*}_{1} x \cup e^{*}_{2} y \cup \lambda_{-1}(\mathcal{R}(g_1, g_2)^{\vee})\right) \right] \\
& =  g^* (x \star y)
\end{align*} 
for any $x,\, y \in K^*(I_G(X))$. The same argument works for $\CH^*(I_G(X), \bullet)$.
\end{proof}
\begin{definition}[Orbifold theories]
Let $X$ and $G$ be as before and  $\mathcal{X}= [X/G]$ be the quotient stack. The \emph{small orbifold $K$-theory} and \emph{small orbifold higher Chow ring} of $\mathcal{X}$ are defined by the $G$-invariant subalgebra of the corresponding stringy theories:
$$
K_{\bullet}^{\orb}(\mathcal{X}): = K_{\bullet}(X, G)^G
$$ 
$$
\CH^*_{\orb}(\mathcal{X}, \bullet): = \CH^*(X, G, \bullet)^G
$$
with the orbifold products are defined as the restriction of the stringy products on the $G$-invariants.
\end{definition}
We will show later (Proposition \ref{prop:independence}) that these definitions are independent of the choice of representations of $\mathcal{X}$.
Note that the stringy higher Chern character induces a ring isomorphism between these two small orbifold theories.

\begin{proposition}
The small orbifold products on orbifold $K$-theory and orbifold higher Chow ring are graded commutative. 
\end{proposition}
\begin{proof}
We only prove the statement for the orbifold $K$-theory, the proof for orbifold higher Chow groups is similar. Just as in \cite[Theorem 1.30]{FG1}, it is enough to show that for any $g, h\in G$, any $x\in K_{\bullet}(X^{g})$ and $y\in K_{\bullet}(X^{h})$, we have the twisted commutativity relation:
$$x \star y=y \star h^{*}(x) \text{  in  } K_{\bullet}(X^{gh}),$$
where $h^{*}(x)\in K_{\bullet}(X^{h^{-1}gh})$ is the image of $x\in K_{\bullet}(X^{g})$ via the natural isomorphism $h.: X^{h^{-1}gh}\xrightarrow{\cong}X^{g}$.

To this end, let $i: X^{\langle g, h\rangle}\hookrightarrow X^{gh}$ be the natural inclusion. The following straightforward computation proves the desired equality: 
\begin{align*}
x \star y =&i_{*}\left(x|_{X^{\langle g, h\rangle}}\cdot y|_{X^{\langle g, h\rangle}}\cdot\lambda_{-1}\mathcal{R}(g, h)^{\vee}\right)\\
=&i_{*}\left(y|_{X^{\langle g, h\rangle}}\cdot h^{*}(x)|_{X^{\langle g, h\rangle}}\cdot\lambda_{-1}\mathcal{R}(h, h^{-1}gh)^{\vee}\right)\\
=&y \star h^{*}(x).
\end{align*}
Here the second equality uses the fact that $\mathcal{R}(g,h)= \mathcal{R}(h, h^{-1}gh)$, see \cite[Lemma 1.10]{FG1}.
\end{proof}


\subsection{Realization functors}
For a compact, almost complex manifold $X$ endowed with an action of a finite group $G$ preserving the almost complex structure, the same construction above can be carried over to define 
the stringy topological $K$-theory  $K^\bullet_{top}(X, G)$, the stringy cohomology $H^*(X,G)$ (with stringy products) and the stringy topological Chern character (see \cite[10.2]{JKK1} and \cite{FG1}). 

Let $X$ be a complex algebraic variety. The set of $\mathbb{C}$-points $X(\mathbb{C})$ inherits the classical (analytic) topology.  
The assignment $X \mapsto X(\mathbb{C})$ defines a functor from the category of (projective) complex varieties to the category of (compact) topological spaces, which sends smooth complex varieties to complex manifolds.  
We will write $K^\bullet_{top}(X)$ and $H^*(X)$ for the topological $K$-theory and singular cohomology of $X (\mathbb{C})$, respectively.

The assignment $E \mapsto E(\mathbb{C})$ induces an exact functor from the category of algebraic vector bundles on $X$ to the category of complex topological vector bundles on $X(\mathbb{C})$.
This induces a natural transformation
$$
F^0(-): K_0(-) \to K^{0}_{top}(-)
$$
between contravariant functors with values in commutative rings which preserves multiplication. This is also a natural transformation of covariant functors on the category of 
smooth projective complex varieties (Baum--Fulton--MacPherson's Riemann--Roch \cite{MR0549773}).
More generally,  for any $n$, there is a natural transformation
$$
F^n(-): K_n(-) \to K^{-n}_{top}(-)
$$
with the same functorial properties \cite{FW1}.

Similarly, the assigment $Z \to Z(\mathbb{C})$ and Poincar\'e duality define the cycle class map $\CH^* (X) \to H^{2*}(X)$. It is generalized to define a natural transformation
$$
F': \CH^{*} (-, \bullet) \to H^{2* - \bullet}(-) 
$$
which forms the commutative diagram of natural transformations
\[
\begin{tikzcd} \label{Commutativity}
K_{\bullet}(-) \arrow[r, "F"] \arrow[d,"\mathbf{ch}"]
& K^{\bullet}_{top}(-) \arrow[d, "ch"] \\
\CH^{*}(-, \bullet) \otimes \mathbb{Q} \arrow[r,  "F'"] & H_{top}^{2*-\bullet}(-) \otimes \mathbb{Q},
\end{tikzcd}
\]
where $ch$ is the topological Chern character. The vertical arrows become isomorphisms if we use rational coefficients. 
Moreover, this diagram is compatible with Riemann--Roch transformations on both algebraic and topological sides \cite[Theorem 5.2]{FW1}.

If $G$ is a finite group acting on $X$, it acts on $X(\mathbb{C})$ and preserves the (almost) complex structure. Moreover $X^g(\mathbb{C}) = X(\mathbb{C})^g$ for any $g \in G$. 
Combining all of these compatible properties, and the fact that the obstruction class $\mathcal{R}(\mathbf{g})$ is represented by the obstruction bundle of Fantechi-Göttsche \cite{FG1}, we obtain ring homomorphisms 
$$
F: K_{\bullet}(X,G) \to K^{\bullet}_{top} (X, G)
$$
and
$$
F': \CH^*(X,G, \bullet) \to H^*(X,G)
$$
between stringy theories. These homomorphisms obviously commute with the stringy Chern characters. 

However, for each smooth complex projective variety $X$, the composition
$$
ch \circ F^n=  F'^{n}\circ \mathbf{ch}: K_n(X) \to H^{*}(X) \otimes \mathbb{Q} 
$$
is known to be zero when $n \ge 1$ (see for example \cite{Gillet2}). 
Therefore, all the algebraic elements in the stringy topological $K$-theory $K^*_{top} (X, G)$ (resp. in the stringy cohomology $H^*(X,G)$) only come from the orbifold Grothendieck ring $K_0(X,G)$ (resp. the orbifold Chow ring $\CH^*(X,G)$). 
In other words, these topological invariants do not give much information about our motivic theories.

One way to obtain the information of (rational) higher algebraic $K$-theory and motivic cohomology by means of topological and geometric data is to use the so-called Beilinson (higher) regulator map whose target is the Deligne cohomology.
For more details about the regulator map and its relation with the Beilinson's conjectures on the values of $L$-functions, we refer the reader to \cite{MR0760999} and \cite{MR0860404}.

Let $X$ be a complex projective variety. \textit{The Deligne complex} $\mathbb{Q}_{\mathcal{D}}(p)$ is the complex
$$
\mathbb{Q}_{\mathcal{D}}(p): = (2\pi i)^p \mathbb{Q} \to \mathcal{O}_X \to \Omega_{X}^1 \to \ldots \to \Omega_{X}^{p-1}
$$
of analytic sheaves on the analytic manifold $X(\mathbb{C})$.
\textit{The Deligne cohomology} $H^q_{\mathcal{D}}(X, \mathbb{Q}(p))$ is defined to be the hypercohomology of $\mathbb{Q}_{\mathcal{D}}(p)$, i.e.,
$$
H^q_{\mathcal{D}}(X, \mathbb{Q}(p)): = \mathbb{H}^q(X(\mathbb{C}), \mathbb{Q}_{\mathcal{D}}(p)).
$$

The total Deligne cohomology $H^{\bullet}_{\mathcal{D}}(X, \mathbb{Q}(*)): = \bigoplus_{q,p} H^{q}_{\mathcal{D}}(X, \mathbb{Q}(p))$ forms a ring with the cup product satisfying the graded commutativity,
i.e., $x \cup y = (-1)^{qq'} y \cup x$ if $x \in H^q_{\mathcal{D}}(X, \mathbb{Q}(p))$ and $y \in H^{q'}_{\mathcal{D}}(X, \mathbb{Q}(p'))$.
It is covariantly functorial with respect to proper morphisms and contravariantly functorial with respect to arbitrary morphisms in $\mathbf{Sm}_{\mathbb{C}}$. 

We have the Beilinson (higher) regulator map \cite{MR0760999}:
$$
\rho: K_{n}(X) \to \bigoplus_{p \ge 0} H^{2p-n}_{\mathcal{D}}(X, \mathbb{Q}(p)),
$$
which is known to be the composition of the the higher Chern character map $\mathbf{ch}: K_{n}(X)\to \CH^{*}(X, n)$ constructed by Gillet \cite{Gillet1} and the the higher cycle class map
$$
\tau: \CH^p(X, n)\otimes \mathbb{Q} \to H^{2p-n}_{\mathcal{D}}(X, \mathbb{Q}(p))
$$
constructed by Bloch in \cite{MR0860404} (see \cite{MR2218900} for a refinement). In other words, the following diagram
\[
\begin{tikzcd}
K_n(X) \arrow[r, "\mathbf{ch}"] \arrow[dr, "\rho"]
& \bigoplus_{p} \CH^p(X, n) \otimes \mathbb{Q} \arrow[d, "\tau"]\\
& \bigoplus_p H^{2p-n}_{\mathcal{D}}(X, \mathbb{Q}(p)).
\end{tikzcd}
\]
commutes. 

The higher cycle class map $\tau$ is co- and contravariantly functorial and commutes with cup product.
Similar to the higher Chern character $\mathbf{ch}$, the Beilinson regulator $\rho$ is contravariantly functorial and commutes with cup product, but not covariantly functorial. 
The lack of commutativity of $\rho$ with taking push-forward is corrected by the Todd class of the tangent bundle. 
If $f: X \to Y$ is a proper morphism between smooth projective varieties, then the following digram commutes
\[
\begin{tikzcd}
K_{\bullet}(Y) \otimes \mathbb{Q} \arrow[r, "f_{*}"] \arrow[d,"\mathbf{td}(TY)\rho"]
& K_{\bullet}(X) \otimes \mathbb{Q} \arrow[d, "\mathbf{td}(TX) \rho"] \\
H^{2*-\bullet}_{\mathcal{D}}(Y, \mathbb{Q}(*)) \arrow[r,  "f_{*}"] & H^{2*-\bullet}_{\mathcal{D}}(X, \mathbb{Q}(*))
\end{tikzcd}
\]

For any vector bundle $V$ on $X$, the top Chern class $c^{\mathcal{D}}_{top}(V)$ of $V$ is defined to be the element $\tau (c_{top}V)$ in $H^{2*}_{\mathcal{D}}(X, \mathbb{Q}(*))$.  
It is well-known that the Deligne cohomology satisfies projection formula and excess intersection formula.

Let $G$ be a finite group acting on $X$, the construction in the previous subsections is applied to define the following
\begin{definition}
The \emph{stringy Deligne cohomology} is
$$
H^{\bullet}_{\mathcal{D}} (X, G, \mathbb{Q}(*)): = H^{\bullet}_{\mathcal{D}}(I_G(X), \mathbb{Q}(*)).
$$
The \emph{stringy product} on the stringy Deligne cohomology is
$$
\alpha \star_{\mathcal{D}} \beta: = \mu_*(e_1^*\alpha \cup e_2^*\beta \cup c^{\mathcal{D}}_{top}(\mathcal{R})).
$$
\end{definition}

The calculation in Section \ref{sect:FiniteGroupQuotient} carries over to show that this stringy product is associative, compatible with the $G$-action, and when restricting to the $G$-invariant part, it is graded commutative. 
Moreover, by mimicking the proof of Theorem \ref{thm:stringChern}, we obtain
\begin{theorem}
\begin{itemize}
\item The higher cycle class map 
$$
\tau: \CH^{*}(X,G, \bullet) \to H^{2*-\bullet}_{\mathcal{D}}(X,G, \mathbb{Q}(*))
$$
is a ring homomorphism with respect to the stringy product.
\item 
Define the stringy regulator 
$$
\mathfrak{p}: K_{\bullet}(X,G) \to H^{2*-\bullet}_{\mathcal{D}}(X,G, \mathbb{Q}(*))
$$
by the formula
$$
\mathfrak{p}(x_g): = \rho(x_g) \cup \mathbf{td}^{-1}(\Im_g)
$$
for any $x_g \in K_{\bullet}(X^g) \otimes \mathbb{Q}$. Then $\mathfrak{p}$ is a ring homomorphisms with respect to the stringy products.
\end{itemize}
\end{theorem}
\begin{remark}
We can replace the Deligne cohomology in the above discussion by the absolute Hodge cohomology. The results hold without any change.
\end{remark}

\section{Orbifold theories: general setting}\label{sect:General}
In this section, we will generalize the orbifold theories constructed in the previous section where a finite group action is considered, to the case of a proper action by a linear algebraic group.
Namely, we assume that $G$ is a complex algebraic group acting on a complex algebraic variety $X$. Equivariant algebraic $K$-theory and equivariant higher Chow groups are considered with rational coefficients. 
Our approach is the one in \cite{EJK1} using twisted pull-backs.


\subsection{Set-up and decomposition into sectors}

For any natural number $n$, the $n$-th inertia variety $I^{n}_{G}(X)$ is defined in \S\ref{subsect:HigherInertia}. Let $\mathcal{X}$ be the quotient Deligne--Mumford stack $[X/G]$, then its \emph{inertia stack} $\mathcal{I}_{\mathcal{X}}:=\mathcal{X}\times_{\mathcal{X}\times \mathcal{X}}\mathcal{X}$ is canonically identified with the following quotient stack 
$$
\mathcal{I}_{\mathcal{X}}= [I_G(X)/G].
$$

Similarly, the \emph{$n$-th inertia stack} of $\mathcal{X}$, which is by definition $\mathcal{I}^{n}_{\mathcal{X}}:=\underbrace{\mathcal{I}_{\mathcal{X}}\times_{\mathcal{X}}\cdots\times_{\mathcal{X}}\mathcal{I}_{\mathcal{X}}}_{n}=\underbrace{\mathcal{X}\times_{\mathcal{X}\times \mathcal{X}}\cdots\times_{\mathcal{X}\times \mathcal{X}}\mathcal{X}}_{n}$, is identified with the quotient stack 
$$
\mathcal{I}^n_{\mathcal{X}} = [I^n_G(X)/G].
$$
Note that for any $n$, the inertia stack $\mathcal{I}^{n}_{\mathcal{X}}$ is independent of the choice of the presentation of $\mathcal{X}$.

\begin{definition}
A diagonal conjugacy class is an equivalence class in $G^{n}$ for the action of $G$ given by $h.(g_{1}, \ldots, g_{n}):=(hg_{1}h^{-1}, \ldots, hg_{n}h^{-1})$. 
We will denote the diagonal conjugacy class of $(g_1, \ldots, g_n)$ by $\{(g_1, \ldots, g_n) \}$.
For any diagonal conjugacy class $\Psi$ in $G^n$, we set
$$
I(\Psi): = \{(g_1, \ldots, g_n ,x) ~\mid ~g_1x = \ldots = g_nx = x \text{ and } (g_1, \ldots, g_n) \in \Psi \} \subset I^n_G(X).
$$
\end{definition}

If $G$ acts quasi-freely, then $I(\Psi) = \emptyset$ unless $\Psi$ consists of elements of finite order. Since we are working over an algebraically closed field of characteristic zero, we have
\begin{proposition}[{\cite[Proposition 2.17 and Lemma 2.27]{EJK1}}] \label{prop:l.c.i}
If $G$ acts quasi-freely on $X$, then $I^n_{G}(X)$ is the disjoin union of finitely many $I(\Psi)$'s. Moreover, each $I(\Psi)$ is smooth if $X$ is smooth.
In particular, the projection $\pi: I(\Psi)\to X$ is a finite l.c.i morphism.
\end{proposition}

In particular, the inertia variety $I_G(X)$ contains $X = X^{\{1\}}$ as a connected component. This is called the \textit{non-twisted sector}. The other components $I_G(\Psi)$ are called the \textit{twisted sectors}. 

\begin{remark}
All the structure maps considered in the case of finite groups (Section \ref{subsect:HigherInertia}) fit well into this generalization. 
For any $\mathbf{g} \in G^n$ with the diagonal conjugacy class $\Psi$, we replace each fixed locus $X^{\mathbf{g}}$ by the component $I(\Psi)$.  
The only difference is that the face and the evaluation maps are no longer inclusions but finite l.c.i morphisms by the above Proposition.
\end{remark}

For any $g \in G$, denote $Z_G(g)$ the centralizer of $g$ in $G$. For $\mathbf{g}: = (g_1, \ldots, g_n) \in G^n$, let $Z_G(\mathbf{g}): = \bigcap_{i=1}^n Z_G(g_i)$ and $X^{\mathbf{g}} = \bigcap_{i=1}^n X^{g_i}$ with the reduced scheme structure.  
\begin{proposition} \label{Morita}
We have the decompositions
$$
\CH^{*}(\mathcal{I}^n_{\mathcal{X}}, \bullet) = \CH^{*}_G(I^n_G(X), \bullet) \cong \bigoplus_{\Psi}\CH^*_{Z_G(\mathbf{g})}(X^{\mathbf{g}}, \bullet)
$$
$$
K_{\bullet}(\mathcal{I}^n_{\mathcal{X}}) = K^G_\bullet(I^n_G(X)) \cong \bigoplus_{\Psi}K^{Z_G(\mathbf{g})}_{\bullet}(X^{\mathbf{g}})
$$
where $\Psi$ runs over all diagonal conjugacy classes of $G^n$ such that $I(\Psi) \neq \emptyset$ and $\mathbf{g}$ is a representative for each $\Psi$.
\end{proposition}
\begin{proof}
We prove the proposition only in the case of $n =1$ to simplify the notation. The proof for an arbitrary $n$ is similar.

Consider firstly the case of higher $K$-theory.
Since $I_G(X) = \coprod_{\Psi}I(\Psi)$ and $G$ acts on each $I(\Psi)$ under this decomposition, we have 
$$
K_{\bullet}^{G}(I_G(X)) = \bigoplus_{\Psi} K_{\bullet}^{G}(I(\Psi)).
$$
So, we only need to prove that
\begin{equation} \label{eq:Morita}
K_{\bullet}^{G}(I(\Psi)) = K_{\bullet}^{Z_G(g)}(X^g)
\end{equation}
for any $g \in \Psi$.

Define the action of $G \times Z_G(g)$ on $G \times X^g$ by the formula 
$$
(h,z).(k,x) \colon = (hkz^{-1}, zx)
$$
and consider $G \times X^h$ as a $G$- and $Z_G(g)$-scheme by identifying the groups $G$ and $Z_G(g)$ with $G \times 1$ and $1 \times Z_G(g) \subset G \times Z_G(g)$, respectively.

The map $G \times X^g \to I(\Psi), (g,x) \mapsto g.x$ is obivously a $G$-equivariant map and a $Z_G(g)$-torsor. 
The category of $G$-vector bundles on $I(\Psi)$ is hence equivalent to the category of $G \times Z_G(g)$-vector bundles on $G \times X^g$ \cite[Proposition 6.2]{Thomason1}. Therefore
\begin{equation} \label{Morita1}
K_{\bullet}^{G}(I(\Psi)) = K_{\bullet}^{G \times Z_G(g)}(G \times X^g).
\end{equation}

Similarly, the projection $G \times X^g \to X^g$ is a  $Z_G(g)$-equivariant map and is a $G$-torsor. 
The category of $Z_G(g)$-vector bundles on $X^g$ is equivalent to the category of $G \times Z_G(g)$-vector bundles on $G \times X^g$ and we have
\begin{equation} \label{Morita2}
K_{\bullet}^{Z_G(g)}(X^g) = K_{\bullet}^{G \times Z_G(g)}(G \times X^g).
\end{equation}
The statement for equivariant $K$-theory is followed from \eqref{Morita1} and \eqref{Morita2}.

Now we consider the case of higher Chow groups. 
For each index $i \in \mathbb{N}$, let $V$ be a representation of $G \times Z_G(g)$ and $U \subset V$ such that $G \times Z_G(g)$ acts freely on $U$ and $V-U \subset V$ has codimension at least $i+1$. 
Consider $V$ and $U$ as $G$- and $Z_G(g)$-sets in the obvious way. We have
\begin{align*}
\CH^i_{Z_G(g)}(X^g, n): & = \CH^i\left((X^g \times U)/Z_G(g), n\right) \\
& = \CH^i((G \times X^g \times U)/(G\times Z_G(g)), n) \\
& = \CH^i((I(\Psi) \times U)/G, n) \\
& = \colon CH^i_{G}(I(\Psi), n).
\end{align*}
The second identity follows from the fact that the projection $G \times X^g \to X^g$ is a $G$-torsor, hence
$$
(G \times X^g \times U) /(G \times Z_G(g)) = [(G \times X^g \times U) /G] /Z_G(g) = (X^g \times U)/Z_G(g).
$$ 
The third identity follows similarly from the fact that $G \times X^g \to I(\Psi)$ is a $Z_G(g)$-torsor.
\end{proof}

\subsection{Twisted pullback and orbifold products}
We first recall the key constructions in \cite{EJK1}.
Assume that $Z$ is an algebraic group acting on $X$. Let $V$ be a $Z$-vector bundle on $X$. 
Let $g$ be an element of finite order acting on the fibers of $V$ such that this action commutes with $Z$-action.  
We define the \textit{logarithmic trace} $L(g)(V)$ by
$$
L(g)(V) = \sum_{k=1}^r \alpha_k V_k \in K_0^{Z}(X)
$$
where $0 \le \alpha_1, \ldots, \alpha_r < 1$ are rational numbers such that $\operatorname{exp}(2 \pi i \alpha_k)$ are the eigenvalues of $g$ and $V_k$ are the corresponding eigenbundles 
(compare to Definition \ref{Logarithmic1}).

\begin{definition}[Twisted pull-backs, \emph{cf.}~\cite{EJK1}]\label{def:TwPullBack}
Keep the same notations as in Proposition \ref{Morita}. For any $n \ge 1$, the \emph{twisted pullback} map is defined to be
$$
f^{tw}: K_0^G(X) \to K^G_0(I^n_G(X)) = \bigoplus_{\Psi} K^{Z_G(\mathbf{g})}_0(X^{\mathbf{g}})
$$
whose $\Psi$-summand is given by
\begin{align*}
f^{tw}_{\Psi} \colon K^G_0(X) & \to K^{Z_G(\mathbf{g})}_0(X^{\mathbf{g}}) \\
V & \mapsto \sum_{i=1}^n L(g_i)(V|_{X^{\mathbf{g}}}) + L(g_1\ldots g_n)^{-1}(V|_{X^{\mathbf{g}}}) + V^{\mathbf{g}} - V|_{X^{\mathbf{g}}},
\end{align*}
where $\mathbf{g}\in \Psi$ is any representative.
\end{definition}
For any any conjugacy class $\Psi$ in $G$, we use the notation $L(\Psi)$ to denote the composition of $L(g)$ with the isomorphism $K^{Z_G(g)}_{\bullet}(X^{g}) \cong K^G_{\bullet}(I(\Psi))$, where $g$ is any element in $\Psi$.
The maps $L(\Psi)$ and $f^{tw}_{\Psi}$ are independent of the choice of representatives.

Let $\mathbb{T}_X:= T_X - \mathfrak{g}$ where $\mathfrak{g}$ is the Lie algebra of $G$. If it will not cause confusion, we will ignore to subscript $X$ to simply write $\mathbb{T}$ for $\mathbb{T}_{X}$. 
Since $\mathbb{T} = \pi^*(T_{\mathcal{X}})$ where $\pi: X \to \mathcal{X}=[X/G]$ is the universal $G$-torsor \cite[Lemma 6.6]{EJK1}, 
$\mathbb{T}$ is a positive element in $K_{0}^{G}(X)$ and its image in $K_0(\mathcal{X})$ is independent of the presentation of $\mathcal{X}$ as a quotient stack.
Since the twisted pullback map $f^{tw}$ takes non-negative elements to non-negative elements \cite[Proposition 4.6]{EJK1}, $f^{tw}(\mathbb{T})$ is a non-negative element in $K^G_0(I^2_G(X))$. 

\begin{definition}
The element $f^{tw}\mathbb{T}$ is called the obstruction bundle and is denoted by $\mathbb{T}^{tw}$. 
For each diagonal conjugacy class $\Psi$ in $G^2$, the restriction of $\mathbb{T}^{tw}$ to $K_0^G(I(\Psi))$ is denoted by $\mathbb{T}^{tw}(\Psi)$. 
\end{definition}

We can now mimic \cite{EJK1} to define the orbifold product for K-theory and higher Chow groups. 
\begin{definition}\label{def:OrbProdGeneral}
Let $e_{1}, e_{2}, \mu: I^{2}_{G}(X)\to I_{G}(X)$ be the three natural morphisms in \S \ref{subsect:HigherInertia}.
\begin{enumerate}[$(i)$]
\item The orbifold product $\star_{c_{\mathbb T}}$ on $\CH^{*}_{G}(I_{G}(X), \bullet)$ is defined by
 $$
 \alpha \star_{c_{\mathbb T}} \beta := \mu_* \left(e_1^*  \alpha \cup e_2^* \beta \cup c_{top}(\mathbb{T}^{tw}) \right).
 $$
 The \emph{orbifold higher Chow ring} of $[X/G]$, denoted by $\CH^{*}_{\orb}([X/G], \bullet)$, is defined to be $\CH^{*}_{G}(I_{G}(X), \bullet)$ equipped with the orbifold product $\star_{c_{\mathbb{T}}}$.
\item The orbifold product $\star_{\mathcal{E}_{\mathbb{T}}}$ on $K_{\bullet}^{G}(I_{G}(X))$ is defined by
 $$
 \alpha \star_{\mathcal{E}_{\mathbb T}} \beta := \mu_* \left(e_1^*  \alpha \cup e_2^* \beta \cup \lambda_{-1}(\mathbb{T}^{tw})^{\vee} \right).
 $$
 The \emph{orbifold K-theory} of $[X/G]$, denoted by $K^{\orb}_{\bullet}([X/G])$, is defined to be $K_{\bullet}^{G}(I_{G}(X))$ equipped with the orbifold product $\star_{\mathcal{E}_{\mathbb T}}$.
\end{enumerate}
\end{definition}
\begin{remark}
On equivariant $K$-theory, push-forwards are $R(G)$-module homomorphisms and pull-backs are $R(G)$-algebra homomorphisms (Section \ref{subsect:EquiK}). 
This implies that the orbifold product $\star_{\mathcal{E}_{\mathbb T}}$ on $K_{\bullet}^{G}(I_{G}(X))$ commutes with the natural action of $R(G)$. 
In particular, $\star_{\mathcal{E}_{\mathbb T}}$ induces a product on the completion $K_{\bullet}^{G}(I_G(X))^{\wedge}$, which is also called orbifold product and is denoted by the same symbol $\star_{\mathcal{E}_{\mathbb{T}}}$.
\end{remark}

\begin{proposition} \label{prop:independence}
The orbifold products $\star_{c_{\mathbb{T}}}$ and $\star_{\mathcal{E}_{\mathbb{T}}}$ are independent of the choice of presentation of $\mathcal{X} = [X/G]$ as a quotient stack.
\end{proposition}
\begin{proof}
We only need to show that the obstruction bundle $\mathbb{T}^{tw}$ in $K_0^G(I^2_G(X)) = K_0(\mathcal{I}^2_{\mathcal{X}})$ is independent of the choice of presentation of $\mathcal{X}$.
This is done by  using again the double filtration argument \cite[Theorem 6.3]{EJK1}.
\end{proof}
\begin{remark}
In the case of a finite group, $\mathbb{T}^{tw}$ is nothing but the obstruction bundle $\mathcal{R}$ in Definition \ref{def:obstruction}. Therefore, under the ring isomorphism $CH^*_G(I_G(X), \bullet) \cong CH^*(I_G(X), \bullet)^G$
(with respect to the usual product), the orbifold product $\star_{c_{\mathbb{T}}}$ is just the orbifold product $\star$ defined in the Section \ref{sect:FiniteGroupQuotient}. 
This implies that the small orbifold product on $\CH^*_{\orb}(\mathcal{X}, \bullet)$ is also independent of the choice of presentation of $\mathcal{X}$ as a quotient stack. 
Since the stringy higher Chern character induces a ring isomorphism $K^{\orb}_{\bullet} (\mathcal{X}) \cong \CH^*_{\orb}(\mathcal{X}, \bullet)$ with respect to orbifold products, we obtain the desired result for 
the small orbifold $K$-theory $K^{\orb}_{\bullet} (\mathcal{X})$.
\end{remark}

\begin{theorem} \label{Associativity3}
The orbifold products $\star_{c_{\mathbb{T}}}$ on $\CH^*_{G}(I_G(X), \bullet)$ and $\star_{\mathcal{E}_{\mathbb T}}$ on $K_{\bullet}^{G}(I_G(X))$ are associative.
\end{theorem}
\begin{proof}

Given $g_1, g_2, g_3 \in G$, let $g_4: = g_1 g_2 g_3$. Let $\Psi_{1,2,3}: = \{(g_1, g_2, g_3)\}$, $\Psi_{12,3}: =\{(g_1g_2, g_3)\}$, $\Psi_{1,23}: = \{(g_1, g_2 g_3)\}$, and so on, be the diagonal conjugacy classes.
Let $e_{1,2} : = e^3_3$, $e_{2,3}: = e^3_1$, $\mu_{12,3}: = f^3_2$ and $\mu_{1,23}: = f^3_3$ be the structure maps defined in the section \ref{subsect:HigherInertia}.
The diagrams
\begin{equation} \label{eq:l.c.i.1}
\begin{tikzcd}
I(\Psi_{1,2,3}) \arrow[r, "e_{1,2}"] \arrow[d, "\mu_{12,3}"]
& I(\Psi_{1,2}) \arrow[d, "\mu"] \\
I(\Psi_{12,3}) \arrow[r,  "e_1"]
& I(\Psi_{12})
\end{tikzcd}
\end{equation}
and
\begin{equation} \label{eq:l.c.i.2}
\begin{tikzcd}
I(\Psi_{1,2,3}) \arrow[r, "e_{2,3}"] \arrow[d, "\mu_{1,23}"]
& I(\Psi_{2,3}) \arrow[d, "\mu"] \\
I(\Psi_{1,23}) \arrow[r, "e_2"]
& I(\Psi_{23})
\end{tikzcd}
\end{equation}
are Cartesian, and the vertical arrows are finite l.c.i morphisms.
Let $E_{1,2}$ and $E_{2,3}$ the excess intersection bundle of \eqref{eq:l.c.i.1} and \eqref{eq:l.c.i.2}, respectively, then
\begin{equation} \label{eq:bundle}
e_{1,2}^* \mathbb{T}^{tw}(\Psi_{1,2}) + \mu_{12,3}^* \mathbb{T}^{tw}(\Psi_{12,3}) + E_{1,2}  
=  e_{2,3}^*\mathbb{T}^{tw}(\Psi_{2,3}) + \mu_{1,23}^* \mathbb{T}^{tw}(\Psi_{1,23}) + E_{2,3} 
\end{equation}
\cite[(26)]{EJK1}.

Consider the diagram
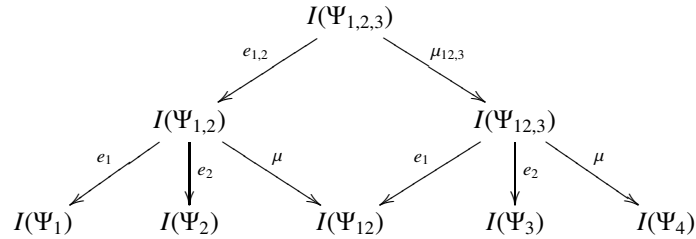
\begin{equation*}
\xymatrix{
&&I(\Psi_{1,2,3}) \ar[dl]_{e_{1,2}} \ar[dr]^{\mu_{12,3}}&&\\
& I(\Psi_{1,2}) \ar[dl]_{e_1} \ar[d]^{e_2} \ar[dr]^{\mu}& &I(\Psi_{12,3})\ar[dl]_{e_1} \ar[d]^{e_2} \ar[dr]^{\mu}&\\
I(\Psi_1)&I(\Psi_2)&I(\Psi_{12})&I(\Psi_3)&I(\Psi_{4})
}
\end{equation*}
where the middle rhombus is the diagram \eqref{eq:l.c.i.1}. Denote $j_i: I(\Psi_{1,2,3}) \to I(\Psi_i)$ for the obvious morphisms with $i= 1, \ldots, 4$.

Let $x \in K^G_*(I(\Psi_1))$, $y \in K^G_*(I(\Psi_2))$ and $z \in K^G_*(I(\Psi_3))$. We have 
\begin{equation}
\begin{split}
& (x \star_{\mathcal{E}_{\mathbb{T}}} y) \star_{\mathcal{E}_{\mathbb{T}}} z\\
  = &
j_{4*}\left(j_1^*x \cup j_2^*y \cup j_3^*z \cup \lambda_{-1}(e_{1,2}^*\mathbb{T}^{tw}(\Psi_{1,2})^{\vee}) \cup \lambda_{-1}( \mu_{12,3}^* \mathbb{T}^{tw}(\Psi_{12,3})^{\vee}) \cup \lambda_{-1}( E_{1,2}^{\vee}) \right) \\
=& j_{4*}\left(j_1^*x \cup j_2^*y \cup j_3^*z \cup \lambda_{-1}(e_{2,3}^*\mathbb{T}^{tw}(\Psi_{2,3})^{\vee}) \cup \lambda_{-1}( \mu_{1,23}^* \mathbb{T}^{tw}(\Psi_{1,23})^{\vee}) \cup \lambda_{-1}( E_{2,3}^{\vee}) \right) \\
 = &x \star_{\mathcal{E}_{\mathbb{T}}} (y \star_{\mathcal{E}_{\mathbb{T}}} z)
\end{split}
\end{equation}
where the second identity follows from \eqref{eq:bundle}, the first and the third one follow from the analogous calculation in the proof of Theorem \ref{Associativity1}, replacing the projection formula by the equivariant projection formula \cite[Corollary 5.8]{Thomason1} and 
the excess intersection formula by the equivariant excess intersection formula \cite[Theorem 3.8]{Kock1} for finite l.c.i morphisms.

The same argument works for equivariant higher Chow groups, where equivariant projection formula and excess intersection formula for finite l.c.i morphisms are direct consequences of the corresponding formulas in the non-equivariant setting.
\end{proof}

\begin{proposition}\label{prop:Commutativity}
The orbifold products are graded commutative with identity.
\end{proposition}
\begin{proof}

It is straightforward to check that the identity for $\star_{\mathcal{E}_{\mathbb{T}}}$ is $[\mathcal{O}_{I_G(\{1\})}] \in K_0^G(I_G(X))$ and the identify for $\star_{c_{\mathbb{T}}}$ is $[I_G(\{1\})] \in \CH^{*}_G(I_G(X))$.

We prove the graded commutativity for $\star_{\mathcal{E}_{\mathbb{T}}}$. The case of $\star_{c_{\mathbb{T}}}$ is similar.

Let $i: I^2_G(X) \to I^2_G(X)$ be the involution induced the by involution $i: G^2 \to G^2$ which exchanges the factors. We have
$$
i^*(\mathbb{T}^{tw}(\Psi)) = \mathbb{T}^{tw}(i (\Psi))
$$
for any diagonal conjugacy class $\Psi $ in $G^2$.

Let $\Psi_1 = e_1 (\Psi)$, $\Psi_2 = e_2 (\Psi)$ and $\Psi_3 = \mu (\Psi)$.
Let $\alpha \in K^G_{\bullet}(I(\Psi_1))$ and $\beta \in K^G_{\bullet} (I(\Psi_2))$. The product $\alpha \star_{\mathcal{E}_{\mathbb{T}}} \beta$ has a contribution in $K^G_{\bullet} (I(\Psi_3))$ given by
$$
\mu_*(e_1^*(\alpha) \cup e_2^*(\beta) \cup \lambda_{-1} \mathbb{T}^{tw}(\Psi)).
$$
The product $\beta \star_{\mathcal{E}_{\mathbb{T}}} \alpha$ has a contribution in $K^G_{\bullet} (I(\Psi_3))$ given by
$$
\mu_*(e_1^*(\beta) \cup e_2^*(\alpha) \cup \lambda_{-1} \mathbb{T}^{tw}(i (\Psi))).
$$

Since $e_1 \circ i = e_2$, $e_2 \circ i = e_1$ and the diagram 
\[
\begin{tikzcd}
I^2_G(\Psi) \arrow[r, "\mu"] \arrow[d, "i"]
& I_G(\Psi_3) \arrow[d, "id"] \\
I^2_G(i(\Psi)) \arrow[r, "\mu"]
& I_G(\Psi_3)
\end{tikzcd}
\]
is Cartesian, we have 
\begin{align*}
\mu_*(e_1^*(\beta) \cup e_2^*(\alpha) \cup \lambda_{-1} \mathbb{T}^{tw}(i (\Psi))) =
 & \mu_*(i^*e_2^*(\beta) \cup i^*e_1^*(\alpha) \cup \lambda_{-1} i^* \mathbb{T}^{tw}(\Psi)) \\
= & \mu_* \left(i^*(e_2^*(\beta) \cup e_1^*(\alpha) \cup \lambda_{-1} \mathbb{T}^{tw}(\Psi)) \right) \\
= & \mu_*(e_2^*(\beta) \cup e_1^*(\alpha) \cup \lambda_{-1} \mathbb{T}^{tw}(\Psi)).
\end{align*}
\end{proof}

Similarly to \cite[Definition 7.3]{EJK1}, we introduce the following notion which relates two (higher) orbifold theories. 
\begin{definition}
The \emph{orbifold Chern character} is the map
$$
\mathfrak{ch}: K_{\bullet}^{G}(I_G(X)) \to  \CH^*_{G}(I_G(X), \bullet)
$$
given by the formula
$$
\mathfrak{ch}(\mathcal{F}_{\Psi}): = \mathbf{ch}(\mathcal{F}_{\Psi}) \mathbf{td}(-L(\Psi)(\mathbb{T}))
$$
for any $\mathcal{F}_{\Psi} \in K_{\bullet}^{G}(I(\Psi))$. 
\end{definition}

\begin{theorem} \label{thm:completion}
The map $\mathfrak{ch}: K_{\bullet}^{G}(I_G(X)) \to \CH^*_{G}(I_G(X), \bullet)$ is a ring homomorphism with respect to orbifold products $\star_{\mathcal{E}_{\mathbb{T}}}$ and $ \star_{c_{\mathbb{T}}}$.
Moreover, this map factors through the completion and gives rise to an isomorphism
\begin{equation} \label{Orbifold Chern Completion}
\mathfrak{ch}: K_{\bullet}^{G}(I_G(X))^{\wedge} \to \CH^*_{G}(I_G(X), \bullet).
\end{equation}
\end{theorem}

\begin{proof}
Given conjugacy classes $\Psi_1$, $\Psi_2$ in $G$ and elements $\alpha_1 \in K^G_{\bullet} (I(\Psi_1))$, $\alpha_2 \in K^G_{\bullet} (I(\Psi_2))$, 
let $\Psi_{1,2}$ be a diagonal conjugacy class in $G^2$ such that $e_1(\Psi_{1,2}) = \Psi_1$, $e_2(\Psi_{1,2}) = \Psi_2$. Let $\Psi_{12}: = \mu(\Psi_{1,2})$. 
By a similar argument as in the proof of Theorem \ref{thm:stringChern}, we have that $\mathfrak{ch} (\alpha_1 \star_{\mathcal{E}_{\mathbb{T}}} \alpha_2) $ is equal to 
$$ \mu_{*}[e_1^* \mathbf{ch}(\alpha_1) e_2^* \mathbf{ch}(\alpha_2) 
c_{top} (\mathbb{T}^{tw}(\Psi_{1,2})) \mathbf{td} (T_{I(\Psi_{1,2})} -\mu^* T_{I(\Psi_{12})} - \mathbb{T}^{tw}(\Psi_{1,2}) - \mu^* L(\Psi_{12})(\mathbb{T})]
$$
while $\mathfrak{ch}(\alpha_1) \star_{c_{\mathbb{T}}} \mathfrak{ch}(\alpha_2) $ is equal to 
$$
\mu_* [ e_1^* \mathbf{ch}(\alpha_1) e_2^* \mathbf{ch}(\alpha_2) c_{top}(\mathbb{T}^{tw}(\Psi_{1,2}))
\mathbf{td}(-e_1^*L(\Psi_1)(\mathbb{T}) - e_2^* L (\Psi_2)(\mathbb{T})].
$$
The identity
$$
T_{I(\Psi_{1,2})} -\mu^* T_{I(\Psi_{12})} - \mathbb{T}^{tw}(\Psi_{1,2}) - \mu^* L(\Psi_{12})(\mathbb{T}) = -e_1^*L(\Psi_1)(\mathbb{T}) - e_2^* L (\Psi_2)(\mathbb{T})
$$
in $K_{\bullet}^G(I(\Psi_{1,2}))$ (compare to \eqref{eq:compare}) is verified in \cite[Equation (42)]{EJK1}. This yields the desired result.

The factorization \eqref{Orbifold Chern Completion} follows from the corresponding property of the higher Chern character.
The bijectivity follows from the completion theorem (Theorem \ref{Completion}) and the invertibility of $\mathbf{td}(-L(\Psi)(\mathbb{T}))$ in $\CH^*_{G}(I_G(X), \bullet)$. 
\end{proof}

\begin{theorem}
Let $G$ be a group and $X$, $Y$ be smooth projective varieties endowed with $G$-actions. 
Let $f: X \to Y$ be a $G$-equivariant \'etale morphism. Then 
\begin{enumerate}[$(i)$]
\item The pull-backs
$$
f^*: K^G_{\bullet} (I_G(Y)) \to K^G_{\bullet} (I_G(X))
$$
and 
$$
f^*: \CH^{*}_G(I_G(Y), \bullet) \to \CH^{*}_G (I_G(X), \bullet) 
$$
are ring homomorphisms with respect to the orbifold products. Moreover, the orbifold higher Chern character commutes with pull-backs.

\item {\rm(Riemann--Roch)} For any conjugacy class $\Psi$ in $G$, let $I_X(\Psi) \subset I_G(X)$ and $I_Y(\Psi) \subset I_G(Y)$ be the corresponding subvarieties. Then the following diagram commutes
\[
\begin{tikzcd}
K_{\bullet}^{G}(I_X(\Psi)) \arrow[r, "f_{*}"] \arrow[d,"\mathbf{td}^G(T_{I_X(\Psi)})\mathfrak{ch}(-)"]
& K_{\bullet}^{G}(I_Y(\Psi)) \arrow[d, "\mathbf{td}^G(T_{I_Y(\Psi)})\mathfrak{ch}(-)"] \\
\CH^{*}_{G}(I_X(\Psi), \bullet) \arrow[r,  "f_{*}"] & \CH^{*}_{G}(I_Y(\Psi), \bullet)
\end{tikzcd}
\]
\end{enumerate}
\end{theorem}
\begin{proof}
Since $f$ is \'etale, $f^*T_Y = T_X$, hence $f^{*} \mathbb{T}_Y = \mathbb{T}_X$. It is also clear from the construction that $f^*L(\Psi)(\mathbb{T}_Y) = L(\Psi)(\mathbb{T}_X)$ and $f^*\mathbb{T}^{tw}_Y = \mathbb{T}^{tw}_X$.
This implies part $(i)$.

For part $(ii)$, let $\mathcal{F}_{\Psi} \in K^G_{\bullet}(I_X(\Psi))$. By the projection formula and the equivariant Riemann-Roch theorem, we have
\begin{align*}
f_{*}[\mathbf{td}^G(T_{I_X(\Psi)})\mathfrak{ch}(\mathcal{F}_{\Psi})] & = f_{*}[\mathbf{td}^G(T_{I_X(\Psi)})\mathbf{ch}(\mathcal{F}_{\Psi}) \mathbf{td}(-L(\Psi)(\mathbb{T}_X))] \\
& = f_{*}[\mathbf{td}^G(T_{I_X(\Psi)}) \mathbf{ch}(\mathcal{F}_{\Psi}) f^*(\mathbf{td}(-L(\Psi)(\mathbb{T}_Y)))] \\
& = f_{*}[\mathbf{td}^G(T_{I_X(\Psi)}) \mathbf{ch}(\mathcal{F}_{\Psi})] \mathbf{td}(-L(\Psi)(\mathbb{T}_Y)) \\
& = \mathbf{td}^G(T_{I_Y(\Psi)})\mathbf{ch}(f_*\mathcal{F}_{\Psi})) \mathbf{td}(-L(\Psi)(\mathbb{T}_Y)) \\
& = \mathbf{td}^G(T_{I_Y(\Psi)}) \mathfrak{ch}(f_* \mathcal{F}_{\Psi}).
\end{align*}
\end{proof}

\subsection{Orbifold motives}\label{subsect:OrbMot}
In \cite{MHRCKummer}, the orbifold motive of a given global quotient of a smooth projective variety by a finite group is defined, in the category of Chow motives (with fractional Tate twists). Using the constructions of \cite{EJK1}, we can now treat more generally a Deligne--Mumford stack which is the quotient of a smooth projective variety by a linear algebraic group. 

We keep the notation and assumptions from the previous subsections. Let $X$, $G$ and $\mathcal{X}:=[X/G]$ be as before. The motive of $\mathcal{X}$ is defined as an object in the category $\DM_{\mathbb Q}$ of rational mixed motives, see \cite[\S 2.4]{HPL} (the condition of being \emph{exhaustive} is satisfied since we assumed the action is linearizable).

Recall that in Definition \ref{def:TwPullBack}, a class $\mathbb{T}^{tw} \in K_{0}^{G}(I^{2}_{G}(X))=K_{0}(\mathcal{I}^{2}_{\mathcal{X}})$ is constructed. Just as in Definition \ref{def:OrbProdGeneral}, we consider the image of the class $c_{top}(\mathbb{T}^{tw})\in \CH(\mathcal{I}^{2}_{\mathcal{X}})$ via the push-forward by the following proper morphism
$$(e_{1}, e_{2}, \mu): \mathcal{I}^{2}_{\mathcal{X}}\to \mathcal{I}_{\mathcal{X}}\times\mathcal{I}_{\mathcal{X}}\times\mathcal{I}_{\mathcal{X}},$$
which is a class $(e_{1}, e_{2}, \mu)_{*}c_{top}(\mathbb{T}^{tw})\in \CH(\mathcal{I}_{\mathcal{X}}\times \mathcal{I}_{\mathcal{X}}\times \mathcal{I}_{\mathcal{X}})$. It is equivalent to a morphism 
$$M(\mathcal{I}_{\mathcal{X}})\otimes M(\mathcal{I}_{\mathcal{X}})\to M(\mathcal{I}_{\mathcal{X}}).$$
By the same argument in Theorem \ref{Associativity2} and Proposition \ref{prop:Commutativity}, one can show that this endows $M(\mathcal{I}_{\mathcal{X}})$ a commutative associative algebra object structure on $M(\mathcal{I}_{\mathcal{X}})$ and it induces the one on $\CH^{*}_{\orb}(\mathcal{X})$. We call this algebra object the \emph{orbifold motive} of $\mathcal{X}=[X/G]$ and denote it by $M_{\orb}(\mathcal{X})$.

In the case where $\mathcal{X}=[X/G]$ with $G$ being a finite group, we recover the orbifold Chow motive $\h_{\orb}(\mathcal{X})$ constructed in \cite{MHRCKummer}, that is $M_{\orb}(\mathcal{X})$ is canonically isomorphic to $\iota\left(\h_{\orb}(\mathcal{X})\right)$ as algebra objects in $\DM_{\mathbb Q}$, where $\iota: \CHM_{\mathbb Q}^{op}\hookrightarrow \DM_{\mathbb Q}$ is the fully faithful tensor functor constructed in \cite{Voevodsky2} which embeds (the opposite category of) the category of Chow motives to the category of mixed motives. 

\subsection{Non-abelian localization for K-theory}
In this subsection, algebraic $K$-theory and higher Chow groups will be considered with complex coefficients. 
We summarize here the results of Edidin--Graham in \cite{EG2} which identify the equivariant $K$-theory of $X$ with a direct summand of the equivariant $K$-theory of its inertia stack.

For any conjugacy class $\Psi$ of $G$, let $m_{\Psi} \subset R(G) \otimes \mathbb{C}$ be the maximal ideal of representations whose virtual characters vanish on $\Psi$.
If $\Psi = \{h\}$ is the conjugacy class of $h$, we will denote $m_{\Psi}$ by $m_h$.
The ideal $m_1$ is simply the augmentation one.
By \cite[Propsition 3.6]{EG2}, there is a decomposition
\begin{equation} \label{eq:factorization}
K_{\bullet}^G(X)\otimes \mathbb{C} \cong \bigoplus_{\Psi} K_{\bullet}^G(X)_{m_{\Psi}}
\end{equation} 
where the sum runs over the finite number of conjugacy classes $\Psi$ such that $I(\Psi)$ is non-empty.

Let $Z$ be an algebraic group acting on a scheme $X$ (or more general, an algebraic space). Let $H$ be a subgroup of the center of $Z$ consisting of semi-simple elements which act trivially on $X$.
There is a natural action of $H$ on $K^Z_{\bullet}(X) \otimes \mathbb{C}$ described as follows.

For any $Z$-equivariant vector bundle $\mathcal{E}$ and any character $\chi$ of $H$, denote $\mathcal{E}_{\chi}$ to be the sub-vectorbundle of $\mathcal{E}$ whose section are given by
$$
\mathcal{E}_{\chi}(U): = \{s \in \mathcal{E}(U) \mid h.s = \chi(h)s\}.
$$
Then $\mathcal{E} = \bigoplus_{\chi} \mathcal{E}_{\chi}$ is a decomposition of $\mathcal{E}$ into the direct sum of $H$-eigenbundles.
Let $\mathbf{Vect}^{\chi,Z}(X)$ be the category of $Z$-equivariant vector bundles on $X$ such that $h$ acts with eigenvalue $\chi(h)$ for any $h \in H$ and let $K_{\bullet}^{\chi, Z}(X)$ 
the $K$-theory of $\mathbf{Vect}^{\chi,Z}(X)$. We have the decomposition
$$
\mathbf{Vect}(Z,X) \xrightarrow{\sim} \prod_{\chi} \mathbf{Vect}^{\chi}(Z,X), \ \mathcal{E} \mapsto (\mathcal{E}_{\chi})_{\chi}
$$
which is obviously exact. This yields a decomposition on $K$-theory
$$
K_n^{Z}(X) \otimes \mathbb{C} = \bigoplus_{\chi} K^{\chi, Z}_n(X) \otimes \mathbb{C}.
$$
Given that, we define the action of $H$ on $K_n^{Z}(X)\otimes \mathbb{C}$ by
$$
h. \mathcal{E} = (\chi(h)^{-1} \mathcal{E}_{\chi})_{\chi}
$$
for any $h \in H$.

When restricting to a point and to $n=0$, this defines an action of $H$ on the representation ring $R(Z)$ so that
$$
h.m_{\Psi} = m_{h \Psi}.
$$
In particular, $h^{-1}.m_h = m_1$, the augmentation ideal of $R(Z)$.

When $Z$ acts on $X$ with finite stabilizers, the decomposition $K_n^{Z}(X) = \bigoplus_{\Psi}K_n^{Z}(X)_{m_{\Psi}}$ is obviously compatible with this $H$-action and we have
$$
h^{-1}.K_n^{Z}(X)_{m_h} = K_n^{Z}(X)_{m_1}.
$$

In the case $Z = Z_G(h)$ and $X$ is $X^h$  where $h$ is a representative of $\Psi$ in $G$, this yields 
$$
h^{-1}.K_n^{Z_G(h)}(X^h)_{m_h} = K_n^{Z_G(h)}(X^h)_{m_1}.
$$
By \eqref{eq:Morita}, $K_n^{Z_G(h)}(X^h)$ is identified with $K_n^{G}(I(\Psi))$. The intersection $\Psi \cap Z$ decomposes into the union of conjugacy classes in $Z$ such that one of them is the conjugacy class of $h$.
Under the decomposition \eqref{eq:factorization}, the localization $K_n^{Z_G(h)}(X^h)_{m_h}$ corresponds to a summand of $K_{\bullet}^{G}(I(\Psi))_{m_{\Psi}}$, which we denote by $K_{\bullet}^{G}(I_G(X))_{c_{\Psi}}$
 \cite[Proposition 3.8]{EG2}.
The action of $h^{-1}$ is then define an isomorphism
$$
t_{\Psi} \colon K_{\bullet}^{G}(I(\Psi))_{c_{\Psi}} = K_{\bullet}^{G}(I(\Psi))_{m_1}.
$$
It is straightforward to check that this isomorphism is independent of the choice of $h$ in $\Psi$. These maps $t_{\Psi}$ resemble to define an isomorphism
$$
t: \bigoplus_{\Psi} K_{\bullet}^{G}(I(\Psi))_{c_{\Psi}} \to K_{\bullet}^{G}(I_G(X))_{m_1} = \bigoplus_{\Psi} K_{\bullet}^G(I(\Psi))_{m_1}.  
$$

The map $j: = \pi|_{I(\Psi)}: I(\Psi) \to X$ is a finite l.c.i morphism (Proposition \ref{prop:l.c.i}). Let $N_{\Psi}$ be the corresponding relative tangent bundle.
Then $j$ induces an ismorphism on the localization of $K$-groups
$$
j_*: K_{\bullet}^{G}(I(\Psi))_{c_{{\Psi}}} \to K_{\bullet}^{G}(X)_{m_{\Psi}}
$$
satisfying
\begin{equation*} 
\alpha = j_*\left(\frac{j^* \alpha}{\lambda_{-1}(N_{\Psi}^*)}\right)
\end{equation*}
for any $\alpha \in K_{\bullet}^G(X)_{m_{\Psi}}$ \cite[Theorem 3.3]{EG2}.

\begin{definition}
We define
\begin{align*}
\varphi \colon K_{\bullet}^{G}(X)\otimes \mathbb{C} & \to K_{\bullet}^{G}(I_G(X))_{m_1} \\
\alpha_{\Psi} & \mapsto t\left(\frac{j^*\alpha_{\Psi}}{\lambda_{-1}(N^*_{\Psi})} \right)
\end{align*}
for any $\alpha_{\Psi} \in K_{\bullet}^{G}(X)_{m_{\Psi}}$.
\end{definition}

\begin{proposition}
The map $\varphi$ is an isomorphism whose inverse is $\varphi^{-1} = j_* \circ t^{-1}$.
\end{proposition}
\begin{proof}
For any $\alpha_{\Psi} \in K_{\bullet}^G(X)_{m_{\Psi}}$, we have
$$
(j_* \circ t^{-1})(\varphi \alpha_{\Psi}) = j_* \left(\frac{j^* \alpha_{\Psi}}{\lambda_{-1}(N^*_{\Psi})} \right) = \alpha_{\Psi}.
$$
Conversely, for any $\beta_{\Psi} \in K_{\bullet}^{G}(I_G(X))_{m_1}$,
$$
\varphi (j_* t^{-1}(\beta_{\Psi})) = t \left( \frac{j^* j_* (t^{-1}{\beta_{\Psi}})}{\lambda_{-1}(N^*(\Psi))} \right) = t(t^{-1} \beta_{\Psi}) = \beta_{\Psi}
$$
where the second identity follows from the self-intersection formula for $j$ \cite[Theorem 3.8]{Kock1}.
\end{proof}

\begin{definition}
For $\alpha, \, \beta \in K_{\bullet}^{G}(X) \otimes \mathbb{C}$, define
\begin{equation*}
\alpha \star_{\mathbb{T}} \beta: = \varphi^{-1}(\varphi(\alpha) \star_{\mathcal{E}_{\mathbb{T}}} \varphi(\beta)).
\end{equation*}
\end{definition}

\begin{theorem}
The product $\star_{\mathbb{T}}$ on $K_{\bullet}^G(X) \otimes \mathbb{C}$ is associative and graded commutative.
Moreover, the map 
$$
\mathfrak{ch} \circ \varphi: K_{\bullet}^G(X)\otimes \mathbb{C} \to \CH^*_{G}(I_G(X), \bullet) \otimes \mathbb{C} 
$$
is a ring isomorphism with respect to the product $\star_{\mathbb{T}}$ on the left and the product $\star_{c_{\mathbb{T}}}$ on the right. 
\end{theorem}
\begin{proof}
By definition, the map $\varphi$ is an algebra isomorphism with respect to the orbifold products.
The first statement is a direct consequence of the commutativity and associativity of $\star_{\mathcal{E}_{\mathbb{T}}}$.
The second statement follows from Theorem \ref{thm:completion}.
\end{proof}

\section{Application: hyper-K\"ahler resolution conjectures}
The idea originates from theoretic physics. Based on considerations from topological string theory of orbifolds in \cite{MR818423} and \cite{MR851703}, one expects a strong relation between the cohomological invariants of an orbifold and those of its crepant resolution. Some first evidences are given in \cite{MR1672108}, \cite{MR1404917}, \cite{MR2027195}, \cite{MR2069013} on the orbifold Euler number and the orbifold Hodge numbers. Later Ruan put forth a much deeper conjectural picture, among which he has the following Cohomological Hyper-K\"ahler Resolution Conjecture in \cite{MR1941583}. We refer the reader to \cite{MR2234886}, \cite{MR2483931}, \cite{MR3112518} for more sophisticated versions.

	\begin{conj}[Ruan's Cohomological HRC]\label{conj:CHRC}
		Let $\mathcal{X}$ be a compact complex orbifold with underlying variety $X$ being
		Gorenstein. If there is a crepant resolution $Y\to X$ with $Y$ admitting a
		hyper-K\"ahler metric, then we have an isomorphism of graded commutative
		$\mathbb{C}$-algebras\,: $H^*(Y, \mathbb{C})\simeq H^*_{orb}(\mathcal{X},\mathbb{C})$.
	\end{conj}
Here the right-hand side is the orbifold cohomology ring defined in \cite{MR1950941} and \cite{MR2104605}.
Conjecture \ref{conj:CHRC} being topological, we investigate in this section its refined counterpart in algebraic geometry, namely Conjecture \ref{conj:HRC} and its stronger version Conjecture \ref{conj:HRC+}. See Introduction for the precise statements.

\begin{lemma}\label{lemma:KChowEquiv}
In Conjecture \ref{conj:HRC}, $(ii)$ and $(iii)$ are equivalent. More generally, in Conjecture \ref{conj:HRC+}, $(ii)^{+}$ and $(iii)^{+}$ are equivalent. 
\end{lemma}
\begin{proof}
The equivalence of $(ii)$ and $(iii)$ can be found in \cite[Proof of Theorem 1.8]{MHRCKummer}. We only show the equivalence of $(ii)^{+}$ and $(iii)^{+}$ here.
By Theorem \ref{thm:main1} $(v)$, there exists an orbifold (higher) Chern character map, which is an isomorphism of algebras from the completion of orbifold algebraic K-theory to the orbifold higher Chow ring:
$$
\mathfrak{ch}: K_\bullet^{G}(I_G(X))^{\wedge}=K_{\bullet}^{\orb}(\mathcal{X})^{\wedge} \xrightarrow{\simeq}\CH^*_{G}(I_G(X), \bullet)=\CH^{*}_{\orb}(\mathcal{X}, \bullet).
$$
On the other hand, by Theorem \ref{Riemann--Roch1}, we have an isomorphism of algebras 
$$\mathfrak{ch}: K_{\bullet}(Y)\xrightarrow{\simeq} \CH^{*}(Y, \bullet).$$
Combining these two isomorphisms, we see the equivalence between $(ii)^{+}$ and $(iii)^{+}$.
\end{proof}

\begin{proposition}\label{prop:MotiveImpliesChow}
In Conjecture \ref{conj:HRC+}, $(iv)$ implies $(iii)^{+}$, hence also $(ii)^{+}$.
\end{proposition}
\begin{proof}
The following general fact follows in a straight-forward way from the definition: let $M$ and $N$ be two commutative algebra objects in the category $\DM:=\DM(\mathbb{C})_{\mathbb{C}}$ of mixed motives (over complex numbers) with complex coefficients, if $M$ and $N$ are isomorphic as algebra objects, then we have an isomorphism of bigraded algebras $$\Hom_{\DM}\left(M, \mathbb{C}(*)[2*-\bullet]\right)\simeq \Hom_{\DM}\left(N, \mathbb{C}(*)[2*-\bullet]\right)$$
where $\mathbb{C}(*)[\bullet]$ is the motivic complex defining motivic cohomology with complex coefficients. Now it suffices to apply this statement to $M=M_{\orb}(\mathcal{X})$ and $N=\iota\left(\h(Y)\right)=M(Y)$, where $\iota: \CHM_{\mathbb{C}}^{op}\hookrightarrow \DM$ is the fully faithful embedding tensor functor  \cite{Voevodsky2}.
\end{proof}

Therefore, in some sense, among the various hyper-K\"ahler resolution conjectures, the motivic one is the strongest and most fundamental. Invoking the series of works \cite{MHRCKummer}, \cite{MHRCK3} and \cite{McKaySurface}, we deduce Theorem \ref{thm:main2}:
\begin{proof}[Proof of Theorem \ref{thm:main2}]
As the motivic hyper-K\"ahler resolution conjecture is proved in all the cases in the statement by \cite{MHRCKummer}, \cite{MHRCK3} and \cite{McKaySurface} (note that isomorphic algebra objects in $\CHM_{\mathbb{C}}$ are isomorphic algebra objects in $\DM_{\mathbb C}$), Proposition \ref{prop:MotiveImpliesChow}  implies that the other hyper-K\"ahler resolution conjectures also hold in these cases.
\end{proof}

\bibliographystyle{amsplain}

\Addresses
\end{document}